\documentclass[11pt,a4paper]{article}
\usepackage{amsmath, amsfonts, amssymb,amsthm}
\usepackage{a4wide}
\usepackage{parskip}
\usepackage{enumitem}
\usepackage{xcolor}
\usepackage{pstricks}
\usepackage{hyperref}
\usepackage{latexsym}
\usepackage{cite}
\def\QED{\hfill {$\square$}\goodbreak \medskip}

\usepackage[left=1in, right=1in,top=1in,bottom=1in]{geometry}
\allowdisplaybreaks
\newcommand{\Om} {\Omega}

\newcommand{\be} {\begin{equation}}
	\newcommand{\ee} {\end{equation}}
\newcommand{\bea} {\begin{eqnarray}}
	\newcommand{\eea} {\end{eqnarray}}
\newcommand{\Bea} {\begin{eqnarray*}}
	\newcommand{\Eea} {\end{eqnarray*}}

\newcommand{\al} {\alpha}

\newcommand{\de} {\delta}

\newcommand{\la} {\lambda}

\def\R{{\mathbb R}}
\def\N{{\mathcal N}}

\def\R{{\mathbb R}}

\def\N{{\mathbb N}}

\newcommand{\ra} {\rightarrow}
\newcommand{\na} {\nabla}

\numberwithin{equation}{section}

\newtheorem{definition}{Definition}[section]
\newtheorem{theorem}{Theorem}[section]
\newtheorem{rem}{Remark}[section]
\newtheorem{lemma}{Lemma}[section]
\newtheorem{prop}{Proposition}[section]

\numberwithin{equation}{section}

\begin{document}
	\setlength{\abovedisplayskip}{3pt}
	\setlength{\belowdisplayskip}{3pt}
	\date{}
	{\vspace{0.01in}
		\title{Normalized Solutions to the Kirchhoff-Choquard Equations with Combined Growth}
        

		\author{ {\bf Divya Goel$\,^{1,}$\footnote{e-mail: {\tt divya.mat@iitbhu.ac.in}},  Shilpa Gupta$\,^{2,}$\footnote{e-mail: {\tt shilpagupta890@gmail.com}}} \\ $^1\,$Department of Mathematical Sciences, Indian Institute of Technology (BHU),\\ Varanasi, 221005, India\\
        $^2\,$Department of Mathematics and Statistics, Indian Institute of Technology Kanpur,\\Kanpur, 208016, India}
		
		\maketitle
		\begin{abstract}
This paper is devoted to the study of the following nonlocal equation:
\begin{equation*}
-\left(a+b\|\nabla u\|_{2}^{2(\theta-1)}\right) \Delta u =\lambda u+\alpha (I_{\mu}\ast|u|^{q})|u|^{q-2}u+(I_{\mu}\ast|u|^{p})|u|^{p-2}u \ \hbox{in} \  \mathbb{R}^{N}, 
\end{equation*}
with the prescribed norm $ \int_{\mathbb{R}^{N}} |u|^{2}= c^2,$
where $N\geq 3$, $0<\mu<N$,  $a,b,c>0$, $1<\theta<\frac{2N-\mu}{N-2}$, $\frac{2N-\mu}{N}<q<p\leq \frac{2N-\mu}{N-2}$, $\alpha>0$  is a suitably small real parameter, $\lambda\in\R$ is the unknown parameter which appears as the Lagrange's multiplier and $I_{\mu}$ is the Riesz potential. We establish existence and multiplicity results and further demonstrate the existence of ground state solutions under the suitable range of $\alpha$. We demonstrate the existence of solution in the case of $q$ is $L^2-$supercritical and  $p= \frac{2N-\mu}{N-2}$, which is not investigated in the literature till now. In addition, we present certain asymptotic properties of the solutions. To establish the existence results, we rely on variational methods, with a particular focus on the mountain pass theorem, the min-max principle, and Ekeland's variational principle.
			
\noindent \textbf{Key words:} Kirchhoff  type problem; Choquard type problem; Variational methods; Normalized solutions; Combined growth;  Ground state solutions; Cardano's formula
			
\medskip
			
\noindent \textit{2020 Mathematics Subject Classification: 35A15, 35J20, 35J60.} 
\end{abstract}
\maketitle
\section{Introduction}
		
\setcounter{equation}{0}
		
This paper aims to establish the  existence and multiplicity results of  weak solutions for the following non-local problems:
\begin{equation}\label{1.1}
\left\{\begin{array}{ll}
-\left(a+b\|\nabla u\|_{2}^{2(\theta-1)}\right) \Delta u & =\lambda u+\alpha (I_{\mu}\ast|u|^{q})|u|^{q-2}u+(I_{\mu}\ast|u|^{p})|u|^{p-2}u \ \hbox{in} \  \R^{N},  \ \ \ \\
\hspace{2.5 cm} \displaystyle\int_{\R^{N}} | u|^{2}   &= c^2,
\end{array}\right.
\end{equation}
where $N\geq 3$, $0<\mu<N$,  $a,b,c>0$, $1<\theta<\frac{2N-\mu}{N-2}$, $\frac{2N-\mu}{N}<q<p\leq \frac{2N-\mu}{N-2}$, $\alpha>0$  is a suitably small real parameter, $\lambda\in\R$ is the unknown parameter which appears as the Lagrange's multiplier, and $I_{\mu}\ast |u|^{t}=\int_{\R^{N}}\frac{|u(y)|^{t}}{|x-y|^{\mu}}dy$ for any $t\in \{p,q\}$.
		
From the mathematical point of view, equation \eqref{1.1} is a  doubly nonlocal due to the presence of terms  	$\int_{\R^N} |\na u |^2 ~dx$ and $(I_{\mu}\ast|u|^{q})$.  This implies that \eqref{1.1} is not a
pointwise identity. This causes some mathematical difficulties, which make the study of   \eqref{1.1} particularly interesting.  The first way to study \eqref{1.1} is to consider the case where $\la$ is a fixed and assigned parameter or in the presence of additional potential $V(x)$, the existence, multiplicity, and behaviour of the solution have been studied extensively in the past decade; see, for example, \cite{chen2022nodal,rui2020kirchhoff,song2019infinitely}, and references therein.  In this case, the solution to the problem is the critical point of the corresponding functional; however, nothing can be given apriori on the $L^2$-norm of the solutions. The second way is having prescribed $L^2$-norm. 	In contemporary physics,  this is called normalized solutions, and there is a wide interest among physicists in exploring normalized solutions.  This interest arises because of the apparent physical significance; for instance, the mass $\|u\|_{L^{2}(\R^{N})}$ can represent the total number of atoms in Bose-Einstein condensation or the power supply in nonlinear optics framework. 
		
As far as we know, in 1997,  Jeanjean \cite{jeanjean1997existence} started the study of normalized solutions for the semi-linear elliptic equation on the following constraint set 
$$S_{c}=\{u\in H^{1}(\R^{N}):\int_{\Omega} | u|^{2}dx=c^2\}.$$
Consider the following equation for 
\begin{equation}\label{i1}
-\Delta u  =\lambda u+|u|^{p-2}u \ \hbox{in} \  \R^{N},  \quad u \in H^1(\R^N), ~~
\int_{\R^{N}} | u|^{2}dx   = c^2,
\end{equation}
where  $\la \in \R$ is a part of the unknown that appears as a Lagrange multiplier. The associated energy functional is defined on the constrained set $S_{c}$ as 
		$$I(u)=\int_{\R^{N}}|\nabla u|^{2}dx-\frac{1}{p}\int_{\R^{N}} |u|^{p}dx.$$ 
		Each  critical point  $u_c \in S_c $ of the functional $I|_{S_c}$  there exists $\la _c \in \R$, 
		corresponds to a Lagrange multiplier $\la_c \in \R$  such that $(u_c, \la_c)$ solves problem \eqref{i1}.  
The exponent $\overline{p}=2+\frac{4}{N}$ is known as $L^{2}-$critical exponent and serves as a critical threshold for various dynamical characteristics, including distinctions between global existence and blow-up, as well as the stability or instability of ground states. The  problems of type \eqref{i1}  where $p<\overline{p}$ is called the  $L^{2}-$ subcritical problems which are studied by \cite{jeanjean1997existence,cazenave1982orbital,shibata2017new}. The energy functional $I$ is bounded below  $p<\overline{p}$, and one can find the global minimizer of the functional $I$.  The problems of type \eqref{i1} where $p>2+\frac{4}{N}$ are called the  $L^{2}-$supercritical problems, and the corresponding energy functional is unbounded below on $S_{c}$. Hence, one can not find the global minimizer of the functional $I$. In the case of variational methods, one of the main difficulties in dealing with this case is that the boundedness of the Palais Smale sequence can not be obtained on $S_{c}$. In \cite{jeanjean1997existence}, Jeanjean addressed this issue by taking into account the Pohozaev set and proving the boundedness of the Palais Smale sequence, which was the first result in this direction.


			
			
			

The case of the concave-convex combined non-linearities is more complex, as the interaction of the $L^{2}-$subcritical, $L^{2}-$critical and $L^{2}-$supercritical exponents strongly affects the geometry of energy functional. Recently,  Soave \cite{soave2020normalized1} studied  the following problem:
\begin{equation}\label{i2}
\left\{\begin{array}{ll}
-\Delta u & =\lambda u+|u|^{p-2}u+\alpha|u|^{q-2}u\ \hbox{in} \  \R^{N},  \ \ \ \\
\int_{\R^{N}} | u|^{2}dx   &= c^2,
\end{array}\right.
\end{equation}
Where $q$ is $L^{2}-$subcritical/$L^{2}-$critical exponent and $p$ is Soboelv subcritical, i.e.,   $2<q\leq 2+\frac{4}{N}\leq p<2^{*}$. Here, the authors registered the existence of ground state solution when  $2<q\leq 2+\frac{4}{N}$ and $2+\frac{4}{N} <p< 2^*$. Subsequently in \cite{soave2020normalized2}, Soave discussed the problem \eqref{i2} for  $2<q<2^*=p$ and proved the existence of ground state solution with negative energy if  $ q \in (2, 2+\frac{4}{N})$. While for $q \in [2+\frac{4}{N},2^*)$ Soave established the existence of mountain type solution with positive energy. He further discussed the existence and non-existence of solution for $\lambda<0$. Later on,  Jeanjean-Jendrej-Le-Visciglia  \cite{jeanjean2022orbital} and Jeanjean-Le \cite{jeanjean2022multiple}, studied Problem \eqref{i2} for more general cases and solved the open problem mentioned by Soave \cite{soave2020normalized2}. 
			
Due to the presence of the Choquard type non-linearity, Problem \eqref{1.1} is known as a Choquard equation.  One of the main tools to deal with such type of equations is Hardy-Littlewood-Sobolev (HLS) \cite{lieb2001analysis} inequality which is stated below.

\begin{prop}\cite{lieb2001analysis}
Let $t_{1},t_{2}>1$ and $0<\lambda<N$ with $1/t_{1}+1/t_{2}+\lambda/N=2$, $f\in L^{t_{1}}(\R^{N})$ and $g\in L^{t_{2}}(\R^{N})$. Then there exists a sharp constant $C$ independent of $f$ and $g$ such that 
$$\left| \int_{\R^{N}}\int_{\R^{N}}\dfrac{ f(x)g(y)}{|x-y|^\lambda}dx \ dy\right|\leq C \|f\|_{L^{t_{1}}(\R^{N})}\|g\|_{L^{t_{2}}(\R^{N})}.$$
\end{prop}
			
The equations of the Choquard type have undergone extensive examination in the existing literature. For a comprehensive exploration of their physical interpretation and a survey of such equations, we direct readers to \cite{moroz2017guide}. Further, for specific results on the existence of solutions concerning Choquard-type equations, one can find relevant contributions in the works of Moroz and Schaftingen \cite{moroz2013groundstates,moroz2015existence}.

   
   Pertaining to the mass-preserved Choquard problem,  the following version of Gagliardo-Nirenberg inequality is used to deal with Choquard-type problems. 
			
\begin{prop}\cite[Proposition 1.2]{zhang2023existence}
Let $N\geq 1$, $\mu\in(0,N)$ and $r\in \left( \frac{2N-\mu}{N},\frac{2N-\mu}{N-2}\right) $ then there exist $C_r>0$(dependent on $r$) such that
\begin{equation}\label{GN1}
\int_{\R^{N}} (I_{\mu}\ast|u|^{r})|u|^{r}dx\leq C_{r}\|\nabla u\|_{2}^{2r\delta_{r}}\|u\|_{2}^{2r(1-\delta_{r})},
\end{equation}
where
$$\delta_{r}=\frac{N(r-2)+\mu}{2r}.$$
\end{prop}
This proposition shows that the exponent $2+ \frac{2-\mu}{N}$ is the $L^2-$ critical exponent for an equation involving Choquard type of non-linearities.   For this critical exponent, Li-Ye \cite{li2014existence}  extended the results of \cite{jeanjean1997existence}  and proved a normalized solution for the following equation:
\begin{equation*}
-\Delta u+\lambda u =(I_{\mu}\ast F(u))f(u) \ \hbox{in} \  \R^{N},
\end{equation*}
where   $\lambda\in \R$ and $F(t)=\int_{\R^{N}}f(s)ds$. Later,  Yao-Chen-R$\breve{\text{a}}$dulescu-Sun \cite{yao2022normalized}, discussed several existence and non-existence results for the following:
\begin{equation}\label{in1}
\left\{\begin{array}{ll}
\hspace{0.2 cm} -\Delta u+\lambda u& =\alpha (I_{\mu}\ast|u|^{p})|u|^{p-2}u+\beta|u|^{q-2}u\ \hbox{in} \  \R^{N},  \ \ \ \\
\int_{\R^{N}} | u|^{2}dx   &= c^2,
\end{array}\right.
\end{equation}
where $p= \frac{2N-\mu}{N}$ is lower Choquard critical exponent and $2<q\leq 2^{*}$.  Li \cite{li2022standing} also proved various existence results for \eqref{in1} when $p$ is upper Choquard critical exponent and $2<q<\frac{2N+4}{N}$. Subsequently,  Shang-Ma \cite{shang2023normalized} studied the following:
\begin{equation*}
\left\{\begin{array}{ll}
\hspace{1 cm} -\Delta u& =\lambda u+\alpha (I_{\mu}\ast|u|^{p})|u|^{p-2}u+(I_{\mu}\ast|u|^{q})|u|^{q-2}u \ \hbox{in} \  \R^{N},  \ \ \ \\
\int_{\R^{N}} | u|^{2}dx   &= c^2,
\end{array}\right.
\end{equation*}
where $\frac{2N-\mu}{N}< p<\frac{2N-\mu}{N-2}$ and $q$ is the upper Choquard critical exponent. 
Many researchers studied the Choquard-type problems with mass preserved property; we refer to the works of Lei-Yang-Zhang \cite{lei2023sufficient}  (for $L^{2}-$subcritical non-linearity), Li-Ye \cite{li2014existence}, Xia-Zhang \cite{xia2023normalized} (for $L^{2}-$supercritical non-linearity) and Li \cite{li2023nonexistence}, Chen-Chen \cite{chen2023normalized}  (for combined non-linearity).

On the other hand, nonlocal problems for normalized solutions involving Kirchhoff type operator had been studied by Ye \cite{ye2015sharp}, for the first time and considered the following equation:
\begin{equation}\label{e3}
\left\{\begin{array}{ll}
\left(a+b\|\nabla u\|_{2}^{2}\right) \Delta u & =\lambda u+ \mu |u|^{q-2}u+|u|^{p-2}u \ \hbox{in} \  \R^{N},  \ \ \ \\
\hspace{1.6cm} \int_{\R^{N}} | u|^{2}dx   &= c^2,
\end{array}\right.
\end{equation}
where $a>0, \ b>0$,  $\mu =0$ and  $N\leq 3$. In the view of Gagliardo-Nirenberg inequality \cite{weinstein1982nonlinear}, the author observed that the $L^{2}-$critical exponent for such type of problems is $2+\frac{8}{N}$ and obtained sharp existence of global constraint minimizer for $2<p<2+\frac{8}{N}$ and $c>0$.  Further, the author proved the existence of a normalized solution to the problem for $2+\frac{8}{N}<p<2^*$ by employing the concentration-compactness principle. For the combined nonlinearity, Zhang and Han \cite{zhang2022normalized} researched the existence of ground state normalized solution of the above problem with $\mu =1,~ N=3,~ \frac{14}{3}<q<p=6$. For $\mu> 0$ and $N=3$, if $2 < q < \frac{10}{3}$ and $\frac{14}{3} <p<6$, Li and Lou \cite{li2022normalized} proved a multiplicity result for \eqref{e3}, and if $2 < p <\frac{10}{3}<q=6$ or $\frac{14}{3} <q<p\leq 6$,  a ground state solution for \eqref{e3} was established.  For more works on Kirchhoff problems, one can refer \cite{liu2024normalized, ye2015sharp,ye2015existence,zeng2017existence} and references therein. 
            
   While it seems that there are very few works on  Kirchhoff-Choquard problems.  In \cite{liu2019multiple}, Liu investigated the following  N-dimension 
Choquard equation involving kirchhoff type perturbation 
   \begin{equation*}
\left\{\begin{array}{ll}
M\left(\,\int\limits_{\R^N} |\nabla u|^2\,dx\right) \Delta u & =\lambda u+ (I_{\mu}\ast|u|^{p})|u|^{p-2}u \ \hbox{in} \  \R^{N},  \ \ \ \\
\hspace{1.6cm} \int_{\R^{N}} | u|^{2}dx   &= c^2,
\end{array}\right.
\end{equation*}
  where $M(t) = a +bt$. Under different ranges of $p$, the author 
obtained the threshold values separating the existence and nonexistence of
critical points. Furthermore,  Liu investigates the behaviours of the Lagrange multipliers and the energies corresponding to the constrained critical points when $c \ra 0$ and $c \ra \infty$, respectively.  To the best of our knowledge, there is no article that addresses the Kirchhoff problem with $M(t)= a+bt^{\theta-1}$  for the concave-convex problem. 

Motivated by the works described above, in this paper, we study the existence and multiplicity of normalized solutions to problems involving combined convex-concave type Choquard nonlinearity. Precisely, we studied the following problem 
		\begin{equation*}
			\left\{\begin{array}{ll}
				-M\left(\,\displaystyle\int\limits_{\R^N} |\nabla u|^2\,dx\right) \Delta u & =\lambda u+\alpha (I_{\mu}\ast|u|^{q})|u|^{q-2}u+(I_{\mu}\ast|u|^{p})|u|^{p-2}u \ \hbox{in} \  \R^{N},  \ \ \ \\
				\hspace{2.5 cm} \displaystyle\int_{\R^{N}} | u|^{2}dx   &= c^2,
			\end{array}\right.
		\end{equation*}
		where $M(t)= a+bt^{\theta-1}$,   $N\geq 3$, $0<\mu<N$,  $a,b>0$, $1<\theta<2_\mu^*$, $\alpha>0$. 
Most of the literature mentioned above for Kirchhoff problems is for the case when $N =3$, as the case $N\geq 3$ is more complex. The techniques and ideas presented in the literature cannot be extended to the Choquard equations. Using the Gagliardo-Nirenberg inequality,  $A^*=\frac{2-\mu}{N}+2$  and  $B^*=\frac{2\theta-\mu}{N}+2$ is the $L^2$- critical exponent for \eqref{in1} and \eqref{1.1} respectively.  
To successfully investigate the existence and multiplicity results of \eqref{1.1}, we divide the problem into the following four cases based on $p$ and $q$. 
\begin{itemize}
\item[Case I:]  $2_{\mu,*}<q<A^{*}$ and $B^{*}<p<2_{\mu}^{*}$\\
Here $q$ is $L^2-$ subcritical, $p$ is $L^2-$ supercritical and HLS subcritical
\item[Case II :]$2_{\mu,*}<q<A^{*}$ and $p=2_{\mu}^{*}$\\
 Here $q$ is $L^2-$ subcritical,  $p$ is $L^2-$ supercritical and HLS  critical
\item[Case III :]  $B^{*}<q<p<2_{\mu}^{*}$\\
Here $p$ and $q$ are both $L^2-$ supercritical and HLS subcritical
\item[Case IV :]  $B^{*}<q<2_{\mu}^{*}$ and $p=2_{\mu}^{*}$\\
Here, $q$ is   $L^2-$ supercritical and HLS subcritical and $p$ is HLS critical
\end{itemize}
where $2_{\mu,*}=\frac{2N-\mu}{N}, \ 2_{\mu}^{*}=\frac{2N-\mu}{N-2}.$

We established the existence and multiplicity of positive solutions to \eqref{1.1} for all $N\geq 3$. Most of the literature mentioned above for Kirchhoff problems
is for the case when $N = 3$, as the case $N \geq 3$ is more intricate.  Actually, the techniques and ideas
presented in the previous papers cannot be extended to higher dimensions.  We study the problem with the Kirchhoff operator $M(t) =a + bt^{\theta-1}$, for all $\theta\in (1, 2^*_\mu)$ for the cases  I, II \& III. But for case IV, to get the existence of a solution, we have to explicitly find a positive root to the algebraic equation of degree $2^*_{\mu}-1$, which is an open problem. So,  we prove only for $\mu =2$ and $N=3$ which is a cubic equation. For more details, we refer to Lemma \ref{l5.4}. In the next section,
we give some preliminary framework of the problem and state the main results of the article.



\section{Preliminaries and Main Results }

This section of the article is intended to provide the variational setting. Further in this section, we state the main results of the current
article with a short sketch of the proof. We kickstart the section by defining the notion of a weak solution to \eqref{1.1} and defining the functional spaces that we will need in this article. 

Define the set, $H^{1}_{rad}(\R^{N})=\{u\in H^{1}(\R^{N}):u(x)=u(|x|)\}$ and $S_{c,r}=H^{1}_{rad}(\R^{N})\cap S_{c}$.
\begin{definition} 
A function  $u\in H^{1}(\R^{N})$ is a weak solution of  \eqref{1.1} if the following holds 
\begin{align}\label{wf2}
\notag-\left(a+b\|\nabla u\|_{2}^{2(\theta-1)}\right) \int_{\R^{N}} \nabla u \ \nabla  v  & =\alpha \int_{\R^{N}}(I_{\mu}\ast|u|^{q})|u|^{q-2}u v\\
&+\int_{\R^{N}} (I_{\mu}\ast|u|^{p})|u|^{p-2}u v,
 \ \forall v\in H^{1}(\R^{N}).
\end{align}
\end{definition}
The associated energy functional $J_{\alpha}:H^{1}(\R^{N})\rightarrow \R$ corresponding to \eqref{wf2} is given by 
$$J_{\alpha}(u)=\frac{a}{2}\|\nabla u\|^{2}_{2}+\frac{b}{2\theta}\|\nabla u\|^{2\theta}_{2}-\frac{\alpha }{2q}\int_{\R^{N}}(I_{\mu}\ast|u|^{q})|u|^{q}-\frac{1}{2p}\int_{\R^{N}}(I_{\mu}\ast|u|^{p})|u|^{p}.$$
One can easily prove that $J_{\alpha}\in C^{1}$ and  at any point $u\in H^{1}(\R^{N})$, we have
 \begin{align*}
J'_{\alpha}(u)(v)= \left(a+b\|\nabla u\|_{2}^{2(\theta-1)}\right) \int_{\R^{N}} \nabla u \ \nabla v  &-\alpha  \int_{\R^{N}}(I_{\mu}\ast|u|^{q})|u|^{q-2}u v\\
&-\int_{\R^{N}} (I_{\mu}\ast|u|^{p})|u|^{p-2}u v,
 \ \forall \ v\in H^{1}(\R^{N}).
\end{align*}
 We also define the following:  Poho\v{z}aev manifold associated with problem \eqref{1.1} 
 $$\mathfrak{P}_{\alpha}=\{u\in S_{c}:P_{\alpha}(u)=0\},$$
where $P_\al(u)=0$ is the Poho\v{z}aev identity associated with problem \eqref{1.1}  
\begin{equation*}
P_{\alpha}(u)=a\|\nabla u\|^{2}_{2}+b\|\nabla u\|^{2\theta}_{2}-\delta_{q}\alpha \int_{\R^{N}}(I_{\mu}\ast|u|^{q})|u|^{q}-\delta_{p}\int_{\R^{N}}(I_{\mu}\ast|u|^{p})|u|^{p}=0.
\end{equation*}

Let $u\in S_{c}$ and $s\in \R$, we define
$$(s\star u)(x)=e^{\frac{Ns}{2}}u(e^{s}x).$$ It can be observed that $s\star u\in S_{c}$. For any $u\in S_{c}$ and $\alpha\in \R^{+}$, we define the fiber map:
$$E_{u}(s)=J_{\alpha}(s\star u)=\frac{ae^{2s}}{2}\|\nabla u\|^{2}_{2}+\frac{be^{2\theta s}}{2\theta}\|\nabla u\|^{2\theta}_{2}-\frac{\alpha e^{2\delta_{q}qs}}{2q}\int_{\R^{N}}(I_{\mu}\ast|u|^{q})|u|^{q}-\frac{e^{2\delta_{p}ps}}{2p}\int_{\R^{N}}(I_{\mu}\ast|u|^{p})|u|^{p}.$$
Then
$$(E_{u})'(s)=ae^{2s}\|\nabla u\|^{2}_{2}+be^{2\theta s}\|\nabla u\|^{2\theta}_{2}-\alpha\delta_{q}e^{2\delta_{q}qs}\int_{\R^{N}}(I_{\mu}\ast|u|^{q})|u|^{q}-\delta_{p}e^{2\delta_{p}ps}\int_{\R^{N}}(I_{\mu}\ast|u|^{p})|u|^{p}$$
and 
\begin{align*}(E_{u})''(s)=2ae^{2s}\|\nabla u\|^{2}_{2}&+2\theta be^{2\theta s}\|\nabla u\|^{2\theta}_{2}-\alpha2q\delta_{q}^2e^{2\delta_{q}qs}\int_{\R^{N}}(I_{\mu}\ast|u|^{q})|u|^{q}\\
&-2p\delta_{p}^2e^{2\delta_{p}ps}\int_{\R^{N}}(I_{\mu}\ast|u|^{p})|u|^{p}.\end{align*}
\begin{rem}\label{r1}
Let $u\in S_{c}$ and $\alpha\in \R^{+}$. Then $(E_{u})'(s)=0$ if and only if $s\star u\in \mathfrak{P_{\alpha}},$ i.e.,
$$\mathfrak{P_{\alpha}}=\{u\in S_{c}:(E_{u})'(s)=0\}.$$
\end{rem}

Partition, $\mathfrak{P_{\alpha}}=\mathfrak{P_{\alpha}^{+}}\bigcup\mathfrak{P_{\alpha}^{-}}\bigcup\mathfrak{P_{\alpha}^0},$ where
\begin{align*}
\mathfrak{P_{\alpha}^{+}}&=\{u\in \mathfrak{P_{\alpha}}:(E_{u})''(0)>0\},\\
\mathfrak{P_{\alpha}^{-}}&=\{u\in \mathfrak{P_{\alpha}}:(E_{u})''(0)<0\},\\
\mathfrak{P}^{0}_{\alpha}&=\{u\in \mathfrak{P_{\alpha}}:(E_{u})''(0)=0\}.
\end{align*}

\begin{rem}
From the Pohozaev identity, $u$ is the critical point of $J_{\lambda}|_{S_{c}}$ if and only if $u\in \mathfrak{P_{\alpha}}(u)$.	
\end{rem}

\begin{prop}\cite[Lemma 2.7]{yu2023normalized}\label{pro1}
Let $q\in[2_{\mu,*},2_{\mu}^{*})$, $ \{u_{n}\}\subseteq H^{1}(\R^{N})$ such that $u_{n}\rightharpoonup u$ in $H^{1}(\R^{N})$. Then
$$\int_{\R^{N}}(I_{\mu}\ast|u_{n}|^{q})|u_{n}|^{q-2}u_{n} \varphi= \int_{\R^{N}}(I_{\mu}\ast|u|^{q})|u|^{q-2}u \varphi+o_{n}(1), \  \forall\varphi\in H^{1}(\R^{N}).$$
\end{prop}

\begin{prop}\cite[Lemma 2.2]{gao2018brezis}\label{ppro1}
	Let $ \{u_{n}\}\subseteq H^{1}(\R^{N})$ such that $u_{n}\rightharpoonup u$ in $H^{1}(\R^{N})$. Then
	$$\int_{\R^{N}}(I_{\mu}\ast|u_{n}|^{2_{\mu}^{*}})|u_{n}|^{2_{\mu}^{*}}= \int_{\R^{N}}(I_{\mu}\ast|u|^{2_{\mu}^{*}})|u|^{2_{\mu}^{*}}+\int_{\R^{N}}(I_{\mu}\ast|v_{n}|^{2_{\mu}^{*}})|v_{n}|^{2_{\mu}^{*}}+o_{n}(1)$$
	where, $v_{n}=u_{n}-u$.
\end{prop}
Now before stating the main results, we fix the following notations:
\begin{equation}\label{a2}
S_{HL}=\inf_{u\in D^{1,2}(\R^{N})\backslash\{0\}}\left\lbrace\dfrac{\|\nabla u\|_{2}^{2}}{\left( \int_{\R^{N}}(I_{\mu}\ast|u|^{2_{\mu}^{*}})|u|^{2_{\mu}^{*}}\right)^\frac{1}{2_{\mu}^{*}}} \right\rbrace,
\end{equation}
\begin{equation}\label{mu1}
\alpha_{1}=
		\left( \dfrac{b(\theta-q\delta_{q})}{\delta_{p}(p\delta_{p}-q\delta_{q}) C_{p}c^{2p(1-\delta_{p})}}\right) ^{\frac{\theta-q\delta_{q}}{p\delta_{p}-\theta}}\dfrac{b(p\delta_{p}-\theta) }{\delta_{q}(p\delta_{p}-q\delta_{q}) C_{q}c^{2q(1-\delta_{q})}},
\end{equation}
\begin{equation}\label{mu2}
\alpha_{2}= \frac{\kappa q}{C_{q}}\left[\dfrac{a}{c^{2q(1-\delta_{q})+\frac{2p(1-\delta_{p})(1-q\delta_{q})}{(p\delta_{p}-\theta)}}}\left(\dfrac{bp}{\theta C_{p}} \right)^{\frac{1-q\delta_{q}}{p\delta_{p}-\theta}}+\dfrac{2}{c^{2q(1-\delta_{q})+\frac{2p(1-\delta_{p})(\theta-q\delta_{q})}{(p\delta_{p}-\theta)}}}\dfrac{\left( \frac{b}{2\theta}\right) ^{\frac{p\delta_{p}-q\delta_{q}}{p\delta_{p}-\theta}}}{\left(\frac{C_{p}}{2p} \right) ^{\frac{\theta-q\delta_{q}}{p\delta_{p}-\theta}}}  \right],
 \end{equation}
\begin{equation}\label{mu3}
\alpha_{3}=\dfrac{\left( \left( S_{HL}^{\theta+1}4ab\right) ^\frac{2_{\mu}^{*}}{22_{\mu}^{*}-(\theta+1)}\frac{q}{(\theta-q\delta_{q})}\right)^\frac{(\theta-q\delta_{q})}{\theta}\left( \frac{b}{\delta_{q}}\right)^\frac{q\delta_{q}}{\theta} (2_{\mu}^{*}-\theta)}{(2_{\mu}^{*}-q\delta_{q})C_{q}c^{2q(1-\delta_{q})}},
\end{equation}
where,
\begin{equation*}
	C_{p} = \left\{\begin{array}{ll}
		&\hbox{the constant given in \eqref{GN1} if } 2_{\mu,*}<p<2_{\mu}^{*}, \\
		&S_{HL}^{-2_{\mu}^{*}} \hbox{ if } p=  2_{\mu}^{*},
	\end{array}\right.
\end{equation*}
\begin{equation}\label{kappa}
	\kappa= \left( \frac{\theta(\theta-q\delta_{q})(\theta-1)}{p\delta_{p}(p\delta_{p}-q\delta_{q})(p\delta_{p}-1)}\right) ^{\frac{\theta-q\delta_{q}}{p\delta_{p}-\theta}}-\left( \frac{\theta(\theta-q\delta_{q})(\theta-1)}{p\delta_{p}(p\delta_{p}-q\delta_{q})(p\delta_{p}-1)}\right) ^{\frac{p\delta_{p}-q\delta_{q}}{p\delta_{p}-\theta}}.
\end{equation}
\begin{rem}\label{rem1}
One can observe that
\begin{equation*}
r\delta_{r} \ \left\{\begin{array}{ll}
& < \theta \hbox{ if } r< B^{*}, \\
&=\theta \hbox{ if } r=  B^{*},\\
&>\theta \hbox{ if } r>  B^{*},
\end{array}\right.
\end{equation*}
and $\delta_{r}<1.$ Also, $\delta_{2_{\mu}^{*}}=1$.
\end{rem}
\begin{theorem}\label{t1}
Let $2_{\mu,*}<q<A^{*},  B^{*}<p<2_{\mu}^{*}$ and $0<\alpha<\min\{\alpha_{1},\alpha_{2}\}$, where $\alpha_{1}$ and $\alpha_{2}$ are defined in \eqref{mu1} and \eqref{mu2}, respectively. Then the following are true:
\begin{enumerate}
\item[(1)] $J_{\alpha}|_{S_{c}}$ has a critical point, say, $u_{c,\alpha,loc}$ such  that $J_{\alpha}(u_{c,\alpha,loc})=m(c,\alpha)<0$ for some $\lambda_{c,\alpha,loc}<0$. Also, $u_{c,\alpha,loc}$ is the local minimizer of $J_{\alpha}$  on the set $\Upsilon_{t_{0}}=\{u\in S_{c}:\|\nabla u\|_{2}<t_{0}\}$ for some $t_{0}>0$. Also, $u_{c,\alpha,loc}$ is a ground state solution of $J_{\alpha}|_{S_{c}}$ and any ground state solution of $J_{\alpha}|_{S_{c}}$ is a local minimizer of $J_{\alpha}$  on the set $\Upsilon_{t_{0}}$. Moreover, $u_{c,\alpha,loc}$ is positive and radially decreasing.
\item[(2)] $J_{\alpha}|_{S_{c}}$ has another critical point, say, $u_{c,\alpha,m}$ such  that $J_{\alpha}(u_{c,\alpha,m})=\varsigma(c,\alpha)$ for some $\lambda_{c,\alpha,m}<0$ and $\varsigma(c,\alpha)>0$. Moreover, $u_{c,\alpha,m}$ is positive and radially decreasing. 
\item[(3)] If $u_{c,\alpha,loc}\in S_{c}$ is a ground state solution for $J_{\alpha}|_{S_{c}}$ then $m(c,\alpha)\rightarrow 0^{-}, \ \|\nabla u_{c,\alpha,loc}\|_{2}\rightarrow 0$ as $\alpha\rightarrow 0^{+}.$
 \item[(4)] $\varsigma(c,\alpha)\rightarrow m(c,0)$ and $u_{c,\alpha,m}\rightarrow u_{0}$ in $H^{1}(\R^{N})$ as $\alpha\rightarrow 0^{+},$ where $m(c,0)=J_{0}(u_{0})$ and $u_{0}$ is the ground state solution of $J_{0}|_{S_{c}}$.
\end{enumerate}
\end{theorem}

\begin{theorem}\label{t2}
	Let $2_{\mu,*}<q<A^{*}$, $p=2_{\mu}^{*}$ and $0<\alpha<\min\{\alpha_{1},\alpha_{2},\alpha_{3}\}$, where $\alpha_{1}, \alpha_{2}$ and $\alpha_{3}$ are defined in \eqref{mu1}, \eqref{mu2} and \eqref{mu3}, respectively. Then the following are true:
	\begin{enumerate}
		\item[(1)] $J_{\alpha}|_{S_{c}}$ has a critical point, say, $u_{c,\alpha,loc}$ such  that $J_{\alpha}(u_{c,\alpha,loc})=m(c,\alpha)<0$ for some $\lambda_{c,\alpha,loc}<0$. Also,  $u_{c,\alpha,loc}$ is the local minimizer of $J_{\alpha}$  on the set $\Upsilon_{t_{0}}=\{u\in S_{c}:\|\nabla u\|_{2}<t_{0}\}$ for some $t_{0}>0$. Also, $u_{c,\alpha,loc}$ is a ground state solution of $J_{\alpha}|_{S_{c}}$ and any ground state solution of $J_{\alpha}|_{S_{c}}$ is a local minimizer of $J_{\alpha}$  on the set $\Upsilon_{t_{0}}$. Moreover, $u_{c,\alpha,loc}$ is positive and radially decreasing.
\item[(2)] If $u_{c,\alpha,loc}\in S_{c}$ is a ground state solution for $J_{\alpha}|_{S_{c}}$ then $m(c,\alpha)\rightarrow 0^{-}, \ \|\nabla u_{c,\alpha,loc}\|_{2}\rightarrow 0$ as $\alpha\rightarrow 0^{+}.$
	\end{enumerate}
\end{theorem}

\begin{theorem}\label{t3}
Let  $B^{*}<q<p<2_{\mu}^{*}$ and $\alpha>0$. Then the following are true:
\begin{enumerate}
\item[(1)] $J_{\alpha}|_{S_{c}}$ has a critical point via the mountain pass theorem, say, $u_{c,\alpha,m}$ such that $J_{\alpha}(u_{c,\alpha,m})=\varsigma(c,\alpha)>0$. Also,  $u_{c,\alpha,m}$ is a positive radial solution to \eqref{1.1} for some $\lambda_{c,\alpha,m}<0$. Moreover, $u_{c,\alpha,m}\in S_{c}$ is a ground state solution for $J_{\alpha}|_{S_{c}}$. 
\item[(2)] $\varsigma(c,\alpha)\rightarrow m(c,0)$ and $u_{c,\alpha,m}\rightarrow u_{0}$ in $H^{1}(\R^{N})$ as $\alpha\rightarrow 0^{+},$ where $m(c,0)=J_{0}(u_{0})$ and $u_{0}$ is the  ground state solution of $J_{0}|_{S_{c}}$.
\end{enumerate}
\end{theorem}

\begin{theorem}\label{t4}
Let $N=3$, $\theta=2$, $\mu=2$,  $\dfrac{10}{3}<q<p=4$ and $\alpha>0$. If $\frac{a^{2}}{4}-\frac{b^{3}S_{HL}^{4}}{27}>0$ then  $J_{\alpha}|_{S_{c}}$ has a critical point via mountain pass theorem, say, $u_{c,\alpha,m}$ such that $$J_{\alpha}(u_{c,\alpha,m})=\varsigma(c,\alpha)\in \left(0,\left( \frac{b \Lambda^{2} S_{HL}^{2} }{8}+\frac{3a\Lambda S_{HL}}{8}\right) \right)$$
 where
 \begin{equation*}
 	\Lambda=\left(\frac{aS_{HL}}{2}+\sqrt{\frac{a^{2}S_{HL}^{2}}{4}-\frac{b^{3}S_{HL}^{6}}{27}}\right)^{\frac{1}{3}}+\left(\frac{aS_{HL}}{2}-\sqrt{\frac{a^{2}S_{HL}^{2}}{4}-\frac{b^{3}S_{HL}^{6}}{27}}\right)^{\frac{1}{3}}. 
 	\end{equation*} Also,  $u_{c,\alpha,m}$ is a positive radial solution to \eqref{1.1} for some $\lambda_{c,\alpha,m}<0$. Moreover, $u_{c,\alpha,m}\in S_{c}$ is a ground state solution for $J_{\alpha}|_{S_{c}}$. 
\end{theorem}

The proofs of Theroems \ref{t1}-\ref{t4} is inspired from \cite{soave2020normalized1,soave2020normalized2}. 
 A standard methodology to obtain solutions using variational techniques consists of looking for minimizers of the functional.  For Theorem  \ref{t1},  the ground state solution is obtained by finding the local minimizers of the minimization problem $$m(c,\al)= \inf_{u \in \Upsilon_r}J_\al(u)=  \inf_{u \in  P_{\alpha}^{+}}J_\al(u)= \inf_{u \in  P_{\alpha}}J_\al(u)<0.  $$
 For the mountain pass type solution, we utilize the definition of $m(c,\al)$ and min-max principle \cite[Theorem 5.2]{ghoussoub1993duality}. While for Theorem \ref{t2}, we cover the loss of compactness by employing the fact that  $m(c,\al)<0$  for $\al<\al_3$. 
 
Now, the $L^2$-supercritical case is a bit involving and challenging.  Establishing the boundedness of the Palais-Smale sequence for the functionals $J_\al$ is a tedious job. For this, we employ the Jeanjean technique, we prove the existence of a bounded Palais-Smale sequence with $P_\al(u_n) \ra 0 $ as $n \ra \infty$. However,  for any sequence $u_n\rightharpoonup u$ in $H^1(\R^N)$ then we cannot conclude 
$$ \lim_{n\ra \infty} \|\nabla u_n\|_{2}^{2(\theta-1)} \int_{\R^{N}} \nabla u_n \nabla v=  \|\nabla u\|_{2}^{2(\theta-1)}\int_{\R^{N}} \nabla u  \nabla v,  ~\text{ for all } v \in H^1(\R^N). $$ 
To overcome this obstacle, we define $D=\lim\limits_{n\rightarrow\infty} \|\nabla u_{n}\|^{2(\theta-1)}_{2}$ and prove that $D>0$. Now   convergence follows for the subcritical case ($B^{*}<q<p<2_{\mu}^{*}$) but for the critical case i.e,   $B^{*}<q<2_{\mu}^{*}$ and $p=2_{\mu}^{*}$,  this tactic  won't work. We have to find an explicit positive solution to the following algebraic equation: 
$$x^{2^*_{\mu}-1}- bS_{HL}^{\theta} x^{\theta-1} - aS_{HL} =0 $$
By analyzing the equation, we can ensure the existence of a positive solution of equation but to the best of our knowledge, finding an explicit solution is an open problem. This paper tries to partially solve this problem in $N=3$ case. Precisely, by using Cardano's formula, we find the exact solution to the above algebraic equation for $N=3$, $\theta=2$, $\mu=2$, and $\frac{a^{2}}{4}-\frac{b^{3}S_{HL}^{4}}{27}>0$. Utilizing this, we prove the existence of a ground state solution for the case $\dfrac{10}{3}<q<p=4$. To the best of our knowledge, till now, there has been no attempt to investigate the existence and multiplicity for \eqref{1.1} for any $\theta$ and $N$.


\begin{rem}
Theorem \ref{t4} continues to hold when $N=3$, $\theta=3$, $\mu=1$,  $\dfrac{9}{2}<q<p=5$ and $\alpha>0$, i.e. $J_{\alpha}|_{S_{c}}$ has a critical point via mountain pass theorem, say, $u_{c,\alpha,m}$ such that $J_{\alpha}(u_{c,\alpha,m})=\varsigma(c,\alpha)\in \left(0,\left( \frac{b \Lambda^{3} S_{HL}^{3}}{15}+\frac{2a\Lambda S_{HL}}{5}\right) \right)$
 where
$$\Lambda=\sqrt{\dfrac{bS_{HL}^{3}+\sqrt{b^2S_{HL}^{6}+4aS_{HL}}}{2}}.$$ Also,  $u_{c,\alpha,m}$ is a positive radial solution to \eqref{1.1} for some $\lambda_{c,\alpha,m}<0$. Moreover, $u_{c,\alpha,m}\in S_{c}$ is a ground state solution for $J_{\alpha}|_{S_{c}}$.
\end{rem}

We organize the rest of the paper as 
 follows: Section 3 studies the boundedness and compactness of the Palais-Smale sequence, while Section 4 is devoted to the technical results involving the geometry of the functional. Moreover, we prove the   Theorems \ref{t1} and \ref{t2}. Finally, in Section 5, we consider the $L^2-$supercritical case and prove the  Theorems \ref{t3} and \ref{t4}.

\section{Compactness  of Palais-Smale sequences}
In this section, we prove that $J_{\alpha}|_{S_{c}}$ satisfies the Palais-Smale condition. To prove the desired result, we took the help of a perturbed Pohozaev constraint. 


\begin{lemma}\label{l1}
Let $2_{\mu,*}<q<B^{*}<p<2_{\mu}^{*}$ or $B^{*}<q<p<2_{\mu}^{*}$. Let $\{u_n\}\subseteq S_{c,r}$ be a Palais Smale sequence for the functional $J_{\alpha}|_{S_{c}}$ of level $l\neq 0$ such that $P_{\alpha}(u_n)\rightarrow 0$ as $n\rightarrow\infty$. Then $u_{n}\rightarrow u$ strongly (upto a subsequence) in $H^{1}(\R^{N})$. Moreover, $u\in S_{c}$ and $u$ is the radial solution of \eqref{1.1} for some $\lambda<0.$
\end{lemma}
\begin{proof}
\textbf{Boundedness of $\{u_n\}$:} Since $\{u_n\}$ be a Palais Smale sequence of level $l$ and $P_{\alpha}(u_n)\rightarrow 0$, we have
\begin{align}\label{l1.1}
\notag J_{\alpha}(u_{n})=\frac{a}{2}\|\nabla u_{n}\|^{2}_{2}&+\frac{b}{2\theta}\|\nabla u_{n}\|^{2\theta}_{2}-\frac{\alpha}{2q}\int_{\R^{N}}(I_{\mu}\ast|u_{n}|^{q})|u_{n}|^{q}\\
&-\frac{1}{2p}\int_{\R^{N}}(I_{\mu}\ast|u_{n}|^{p})|u_{n}|^{p}+o_{n}(1)=l
\end{align}
and 
\begin{align}\label{l1.2}
\notag P_{\alpha}(u_{n})=a\|\nabla u_{n}\|^{2}_{2}+b\|\nabla u_{n}\|^{2\theta}_{2}&-\delta_{q}\alpha\int_{\R^{N}}(I_{\mu}\ast|u_{n}|^{q})|u_{n}|^{q}\\
&-\delta_{p}\int_{\R^{N}}(I_{\mu}\ast|u_{n}|^{p})|u_{n}|^{p}=o_{n}(1).
\end{align}
\textbf{Case I:} $2_{\mu,*}<q<B^{*}<p<2_{\mu}^{*}$\\
Using \eqref{l1.1} and \eqref{l1.2}, we obtain
\begin{align*}
\frac{a}{2}\|\nabla u_{n}\|^{2}_{2}&+\frac{b}{2\theta}\|\nabla u_{n}\|^{2\theta}_{2}-\frac{\alpha}{2q}\int_{\R^{N}}(I_{\mu}\ast|u_{n}|^{q})|u_{n}|^{q}\\
-\frac{1}{2p}&\left(\frac{a}{\delta_{p}}\|\nabla u_{n}\|^{2}_{2}+\frac{b}{\delta_{p}}\|\nabla u_{n}\|^{2\theta}_{2}-\frac{\alpha\delta_{q}}{\delta_{p}}\int_{\R^{N}}(I_{\mu}\ast|u_{n}|^{q})|u_{n}|^{q}\right)\leq l+1. 
\end{align*}
Now, employing Gagliardo-Nirenberg inequality \eqref{GN1}, we get

\begin{align*}
\notag \left(\frac{a}{2}- \frac{a}{2p\delta_{p}}\right)\|\nabla u_{n}\|^{2}+\left(\frac{b}{2\theta}- \frac{b}{2p\delta_{p}}\right)\|\nabla u_{n}\|^{2\theta}_{2}\leq(l+1)&+\frac{\alpha}{2q}\left(1-\frac{q\delta_{q}}{p\delta_{p}}\right)\\
&C_{q}\|\nabla u\|_{2}^{2q\delta_{q}}c^{2q(1-\delta_{q})},
\end{align*}
this gives,
$$\left(\frac{b}{2\theta}- \frac{b}{2p\delta_{p}}\right)\|\nabla u_{n}\|^{2\theta}_{2}\leq(l+1)+\frac{\alpha}{2q}\left(1-\frac{q\delta_{q}}{p\delta_{p}}\right)C_{q}\|\nabla u\|_{2}^{2q\delta_{q}}c^{2q(1-\delta_{q})}.$$
Hence, the conclusion follows from the fact that $q\delta_{q}<\theta$ and $p\delta_{p}>\theta>1$.

\textbf{Case II:} $B^{*}<q<p<2_{\mu}^{*}$ 

Again, by \eqref{l1.1}, \eqref{l1.2} and Remark \ref{rem1}, we have
\begin{align*}
\frac{a}{2}\|\nabla u_{n}\|^{2}_{2}&+\frac{1}{2\theta}\left(-a\|\nabla u_{n}\|^{2}_{2}+\delta_{q}\alpha\int_{\R^{N}}(I_{\mu}\ast|u_{n}|^{q})|u_{n}|^{q}+\delta_{p}\int_{\R^{N}}(I_{\mu}\ast|u_{n}|^{p})|u_{n}|^{p} \right) \\&-\frac{\alpha}{2q}\int_{\R^{N}}(I_{\mu}\ast|u_{n}|^{q})|u_{n}|^{q}-\frac{1}{2p}\int_{\R^{N}}(I_{\mu}\ast|u_{n}|^{p})|u_{n}|^{p}+o_{n}(1)\leq l+1,
\end{align*}
which implies
\begin{align*}
\left(\frac{a}{2}- \frac{a}{2\theta}\right)\|\nabla u_{n}\|^{2}_{2}&+\left(\frac{\delta_{p}}{2\theta}- \frac{1}{2p}\right)\int_{\R^{N}}(I_{\mu}\ast|u_{n}|^{p})|u_{n}|^{p}\\
&+\alpha\left(\frac{\delta_{q}}{2\theta}-\frac{1}{2q}\right)\int_{\R^{N}}(I_{\mu}\ast|u_{n}|^{q})|u_{n}|^{q}\leq l+1.
\end{align*}
It concludes,
$$ \left(\frac{a}{2}-\frac{a}{2\theta}\right)\|\nabla u_{n}\|^{2}_{2}\leq l+1,$$
that is, $\{u_{n}\}$ is bounded in $H^{1}(\R^{N})$.

\textbf{There exists a Lagrange multiplier $\boldsymbol{\lambda_{n}\rightarrow \lambda\in \R}$:} 
As $H^{1}_{rad}(\R^{N})\hookrightarrow L^{r}(\R^{N})$ compactly for all $r\in\left( 2,2^{*}\right)$, there exists $u\in H^{1}_{rad}(\R^{N})$ such that 
$u_{n}\rightharpoonup u$ in $H^{1}(\R^{N})$, $u_{n}\rightarrow u$ in $L^{r}(\R^{N})$ and $u_{n}\rightarrow u$ a.e. in $\R^{N}$.

As $\{u_n\}\subseteq S_{c,r}$ is a Palais Smale sequence for the functional $J_{\alpha}|_{S_{c}}$ then by the Lagrange multiplier rule, there exists $\lambda_{n}\in \R$ such that
\begin{align}\label{l1.3}  
\notag \left(a+b\|\nabla u_{n}\|_{2}^{2(\theta-1)}\right) \int_{\R^{N}} \nabla u_{n} \ \nabla v  & -\lambda_{n} \int_{\R^{N}} u_{n} v-\alpha \int_{\R^{N}}(I_{\mu}\ast|u_{n}|^{q})|u_{n}|^{q-2}u_{n} v\\
&-\int_{\R^{N}} (I_{\mu}\ast|u_{n}|^{p})|u_{n}|^{p-2}u_{n} v=o_{n}(1), \ \forall \ v\in H^{1}_{rad}(\R^{N}).
\end{align}
In particular, substituting $v=u_{n}$ in \eqref{l1.3}, we get
\begin{equation}\label{l1.4}
\lambda_{n}c^{2}=a\|\nabla u_{n}\|^{2}_{2}+b\|\nabla u_{n}\|^{2\theta}_{2}-\alpha\int_{\R^{N}}(I_{\mu}\ast|u_{n}|^{q})|u_{n}|^{q}-\int_{\R^{N}}(I_{\mu}\ast|u_{n}|^{p})|u_{n}|^{p}+o_{n}(1).
\end{equation}

Now, using the  fact that $\{u_{n}\}$ is bounded and employing \eqref{GN1}, there exists $M>0$ such that
$$|\lambda_{n}c^{2}|\leq a\|\nabla u_{n}\|^{2}_{2}+b\|\nabla u_{n}\|^{2\theta}_{2}+\alpha C_{q}\|\nabla u_n\|_{2}^{2q\delta_{q}}c^{2q(1-\delta_{q})} +C_{p}\|\nabla u_n\|_{2}^{2p\delta_{p}}c^{2p(1-\delta_{p})}+o_{n}(1)\leq M,$$
that is, $\lambda_{n}$ is a bounded sequence and  $\lambda_{n}\rightarrow \lambda\in \R,$
up to a subsequence.

$\boldsymbol{\lambda<0}$ \textbf{and} $\boldsymbol{u\neq 0:}$ 
Taking into account,  $P_{\alpha}(u_n)\rightarrow 0$, boundedness of $\{u_n\}$, \eqref{l1.4} and  \eqref{GN1}, one gets
\begin{align}\label{l1.5}
\lambda_{n}c^{2}&=\alpha(\delta_{q}-1)\int_{\R^{N}}(I_{\mu}\ast|u_{n}|^{q})|u_{n}|^{q}
+(\delta_{p}-1)\int_{\R^{N}}(I_{\mu}\ast|u_{n}|^{p})|u_{n}|^{p}+o_{n}(1)\\
\notag  &\leq \alpha(\delta_{q}-1)C_{q}\|\nabla u_{n}\|_{2}^{2q\delta_{q}}c^{2q(1-\delta_{q})}+(\delta_{p}-1)C_{p}\|\nabla u_{n}\|_{2}^{2p\delta_{p}}c^{2p(1-\delta_{p})}+o_{n}(1)\\
\notag &\leq \alpha(\delta_{q}-1)c_{1}c^{2q(1-\delta_{q})}+(\delta_{p}-1)c_{2}c^{2p(1-\delta_{p})}+o_{n}(1),
\end{align}
for some $c_{1},c_{2}>0.$
Taking $n\rightarrow \infty$ and using the fact $\delta_{p},\delta_{q}<1$, we obtain $\lambda c^{2}\leq 0$ which implies $\lambda\leq 0.$

If $\lambda=0$ then $\lambda_{n}\rightarrow 0$. Thus, by \eqref{l1.5}, one gets
$$\int_{\R^{N}}(I_{\mu}\ast|u_{n}|^{q})|u_{n}|^{q}\rightarrow 0 \text{ and } \int_{\R^{N}}(I_{\mu}\ast|u_{n}|^{p})|u_{n}|^{p}\rightarrow 0.$$

By the above and using the fact that $P_{\alpha}(u_n)\rightarrow 0$, we have $J_{\alpha}(u_{n})\rightarrow 0\neq l$, which is a contradiction. Hence, $\lambda_{n}\rightarrow \lambda< 0$ and $u\neq 0.$

\textbf{Conclusion:} Let $D=\lim\limits_{n\rightarrow\infty} \|\nabla u_{n}\|^{2(\theta-1)}_{2}\geq\|\nabla u\|^{2(\theta-1)}_{2}>0.$ Using Proposition \ref{pro1} in \eqref{l1.3}, we get
\begin{align}\label{l2.31}
\notag \left(a+bD\right) \int_{\R^{N}} \nabla u \ \nabla v  & -\lambda \int_{\R^{N}} u v -\alpha \int_{\R^{N}}(I_{\mu}\ast|u|^{q})|u|^{q-2}u v\\
&-\int_{\R^{N}} (I_{\mu}\ast|u|^{p})|u|^{p-2}u v=0, \ \forall \ v\in H^{1}(\R^{N}). 
\end{align}
Testing \eqref{l2.31}-\eqref{l1.3} with $v=u_{n}-u$, we obtain
\begin{align*}\label{l1.4}
	\left(a+b\|\nabla u_{n}\|_{2}^{2(\theta-1)}\right)& \int_{\R^{N}} |\nabla (u_{n}-u)|^{2}   -\lambda \int_{\R^{N}} | u_{n}-u|^{2}\\
	&-\alpha \int_{\R^{N}}[(I_{\mu}\ast|u_{n}|^{q})|u_{n}|^{q-2}u_{n}-(I_{\mu}\ast|u|^{q})|u|^{q-2}u] (u_{n}-u)\\
	&-\int_{\R^{N}} [(I_{\mu}\ast|u_{n}|^{p})|u_{n}|^{p-2}u_{n}-(I_{\mu}\ast|u|^{p})|u|^{p-2}u] (u_{n}-u)=o_{n}(1).
\end{align*}
Now, passing the limit $n\rightarrow\infty$ and using Holder's inequality together with the fact that $u_{n}\rightarrow u$ in $L^{r}(\R^{N})$ for each $2<r<2^{*}$. Hence, 
$$\lim_{n\rightarrow\infty}	\left(a+b\|\nabla u_{n}\|_{2}^{2(\theta-1)}\right)\int_{\R^{N}} |\nabla (u_{n}-u)|^{2} -\lambda\int_{\R^{N}} | u_{n}-u|^{2}=0,$$  
that is, $u_{n}\rightarrow u$ in $H^{1}(\R^{N}),$ which completes the proof.
\QED\end{proof}

\begin{lemma}\label{l2}
Let  $2_{\mu,*}<q<B^{*}<p=2_{\mu}^{*}$ or
 $B^{*}<q<p=2_{\mu}^{*}$. Let $\{u_n\}\subseteq S_{c,r}$ be a Palais Smale sequence for the functional $J_{\alpha}|_{S_{c}}$ of level $l\neq 0$ with 
 \begin{equation}\label{cl}
 l<\left( S_{HL}^{\theta+1}4ab\right) ^\frac{2_{\mu}^{*}}{22_{\mu}^{*}-(\theta+1)}\left(\frac{2_{\mu}^{*}-\theta}{22_{\mu}^{*} \theta} \right) 
 \end{equation}
and $P_{\alpha}(u_n)\rightarrow 0$ as $n\rightarrow\infty$. Then, one of the following is true:
\begin{enumerate}
\item[(1)] either $u_{n}\rightharpoonup u\neq 0$ weakly in $H^{1}(\R^{N})$ but not strongly then for some $\lambda<0$, $u$ is the solution of the following:
$$
-\left(a+bD\right) \Delta u =\lambda u+\alpha (I_{\mu}\ast|u|^{q})|u|^{q-2}u+(I_{\mu}\ast|u|^{2_{\mu}^{*}})|u|^{2_{\mu}^{*}-2}u \ \hbox{in} \  \R^{N} $$
and
$$l-\left( S_{HL}^{\theta+1}4ab\right) ^\frac{2_{\mu}^{*}}{22_{\mu}^{*}-(\theta+1)}\left( \frac{2_{\mu}^{*}-\theta}{22_{\mu}^{*}\theta}\right)\geq \tilde{J}_{\alpha}(u),$$ where $\tilde{J}_{\alpha}(u)=\left( \frac{a}{2}+\frac{Db}{2\theta}\right) \|\nabla u\|^{2}_{2}-\frac{\alpha}{2q}\int_{\R^{N}}(I_{\mu}\ast|u|^{q})|u|^{q}-\frac{1}{22_{\mu}^{*}}\int_{\R^{N}}(I_{\mu}\ast|u|^{2_{\mu}^{*}})|u|^{2_{\mu}^{*}}$ and $D=\lim\limits_{n\rightarrow\infty} \|\nabla u_{n}\|^{2(\theta-1)}_{2}$.
\item[(2)] or $u_{n}\rightarrow u$ strongly in $H^{1}(\R^{N}).$ Also, $J_{\alpha}(u)=l$ and $u$ solves \eqref{1.1} for some $\lambda<0$. 
 Moreover, $u\in S_{c}$ and $u$ is the radial solution of \eqref{1.1} for some $\lambda<0.$
 \end{enumerate}
\end{lemma}
\begin{proof}
Boundedness of $\{u_n\}$ and existence of Lagrange multipliers $\lambda_{n}$ such that $\lambda_{n}\rightarrow \lambda$ with $\lambda\leq0$ follows exactly the same as the proof of Lemma \ref{l1}. 

Next, we will prove that $\lambda<0$ and $u\neq 0$. Using Remark \ref{rem1}, \eqref{l1.5}, the Gagliardo-Nirenberg inequality \eqref{GN1}  and boundedness of $\{u_n\}$, we have 
\begin{align}\label{l1.61}
\lambda_{n}c^{2}&=\alpha(\delta_{q}-1)\int_{\R^{N}}(I_{\mu}\ast|u_{n}|^{q})|u_{n}|^{q}
+o_{n}(1)\\
\notag &\leq \alpha(\delta_{q}-1)C_{q}\|\nabla u_{n}\|_{2}^{2q\delta_{q}}c^{2q(1-\delta_{q})}+o_{n}(1)\\
\notag &\leq \alpha(\delta_{q}-1)c_{3}c^{2q(1-\delta_{q})}+o_{n}(1),
\end{align}
for some $c_{3}>0.$
Therefore, passing the limit $n\rightarrow \infty$ and using the fact $\delta_{q}<1$, we have $\lambda c^{2}\leq 0$ which implies $\lambda\leq 0.$

\textbf{Claim:} $\lambda\neq 0.$ Let on the contrary, $\lambda=0$ then by \eqref{l1.61}, we get
\begin{equation}\label{l1.101}
\lim\limits_{n\rightarrow\infty}\int_{\R^{N}}(I_{\mu}\ast|u_{n}|^{q})|u_{n}|^{q}=0.
\end{equation}
Using the above information with $P_{\alpha}(u_n)\rightarrow 0$, we obtain
\begin{equation}\label{l1.71}
\lim\limits_{n\rightarrow\infty}\left( a\|\nabla u_{n}\|^{2}_{2}+b\|\nabla u_{n}\|^{2\theta}_{2}\right) =\lim\limits_{n\rightarrow\infty}\int_{\R^{N}}(I_{\mu}\ast|u_{n}|^{2_{\mu}^{*}})|u_{n}|^{2_{\mu}^{*}}.
\end{equation}
Let $\lim\limits_{n\rightarrow\infty} \|\nabla u_{n}\|^{2}_{2}=t$ and $\lim\limits_{n\rightarrow\infty}\int_{\R^{N}}(I_{\mu}\ast|u_{n}|^{2_{\mu}^{*}})|u_{n}|^{2_{\mu}^{*}}=m,$ thus form \eqref{l1.71}, we have
\begin{equation}\label{l1.79}
at+bt^{\theta}=m.
\end{equation}
This subsequently gives
$2(abt^{\theta+1})^{1/2}\leq m$, i.e.,
$\lim\limits_{n\rightarrow\infty}\|\nabla u_{n}\|^{2{(\theta+1)}}_{2}\leq \frac{m^2}{4ab}$. By using \eqref{a2}, we get
$S_{HL}^{(\theta+1)}\left( \int_{\R^{N}}(I_{\mu}\ast|u_{n}|^{2_{\mu}^{*}})|u_{n}|^{2_{\mu}^{*}}\right)^\frac{(\theta+1)}{2_{\mu}^{*}}\leq \frac{m^2}{4ab}.$ It implies $S_{HL}^{(\theta+1)}m^\frac{(\theta+1)}{2_{\mu}^{*}}\leq \frac{m^2}{4ab},$ that is
\begin{equation}\label{l1.111}
m\geq\left( S_{HL}^{\theta+1}4ab\right) ^\frac{2_{\mu}^{*}}{22_{\mu}^{*}-(\theta+1)} \text{ and } \lim\limits_{n\rightarrow\infty} \|\nabla u_{n}\|^{2}_{2}\leq \left(\frac{m}{b} \right)^{\frac{1}{\theta}}. \end{equation}
%
%
As $\{u_n\}$ is a Palais-Smale sequence for the functional $J_{\alpha}$ of level $l$, using \eqref{l1.79} and \eqref{l1.111}, we obtain
\begin{align*}
l=\lim\limits_{n\rightarrow\infty}J_{\alpha}(u_{n})&=\lim\limits_{n\rightarrow\infty}\left( \frac{a}{2}\|\nabla u_{n}\|^{2}_{2}+\frac{b}{2\theta}\|\nabla u_{n}\|^{2\theta}_{2} -\frac{1}{22_{\mu}^{*}}\int_{\R^{N}}(I_{\mu}\ast|u_{n}|^{2_{\mu}^{*}})|u_{n}|^{2_{\mu}^{*}}\right)\\
&\geq\lim\limits_{n\rightarrow\infty}\left( \frac{m}{22_{\mu}^{*}}(2_{\mu}^{*}-1)-\left( \frac{b}{2}\right) \left( \frac{\theta-1}{\theta}\right)  \|\nabla u_{n}\|^{2\theta}_{2}\right)\\
&\geq\left( \frac{m}{22_{\mu}^{*}}(2_{\mu}^{*}-1)-\left( \frac{m}{2}\right)\left(  \frac{\theta-1}{\theta}\right)  \right)\\
&\geq\frac{m}{2}\left( \frac{2_{\mu}^{*}-\theta}{2_{\mu}^{*}\theta}\right)\\
&\geq\left( S_{HL}^{\theta+1}4ab\right) ^\frac{2_{\mu}^{*}}{22_{\mu}^{*}-(\theta+1)}\left( \frac{2_{\mu}^{*}-\theta}{22_{\mu}^{*}\theta}\right), 
\end{align*}
which is a contradiction to \eqref{cl}. Hence, our supposition is wrong, therefore $\lambda<0$ and $u\neq 0$.

By considering \eqref{l1.3} and using Proposition \ref{pro1}, we have
\begin{align}\label{l2.3}
\notag \left(a+bD\right) \int_{\R^{N}} \nabla u \ \nabla v  & -\lambda \int_{\R^{N}} u v-\alpha \int_{\R^{N}}(I_{\mu}\ast|u|^{q})|u|^{q-2}u v\\
&-\int_{\R^{N}} (I_{\mu}\ast|u|^{2_{\mu}^{*}})|u|^{2_{\mu}^{*}-2}u v, \ \forall \ v\in H^{1}(\R^{N}) 
\end{align}
which implies that $u\in H^{1}(\R^{N})$ is the solution for the equation
$$
-\left(a+bD\right) \Delta u =\lambda u+\alpha (I_{\mu}\ast|u|^{q})|u|^{q-2}u+(I_{\mu}\ast|u|^{2_{\mu}^{*}})|u|^{2_{\mu}^{*}-2}u \ \hbox{in} \  \R^{N} $$
and the solution to the above problem satisfies the following Pohozaev identity
\[ Q_{\alpha}(u)=\left(a+bD\right)\|\nabla u\|^{2}_{2}-\delta_{q}\alpha\int_{\R^{N}}(I_{\mu}\ast|u|^{q})|u|^{q}-\int_{\R^{N}}(I_{\mu}\ast|u|^{2_{\mu}^{*}})|u|^{2_{\mu}^{*}}=0.\]

Define $v_{n}=u_{n}-u$. Using the fact that $P_{\alpha}(u_n)=o_{n}(1)$ with $Q_{\alpha}(u)=0$, Brezis-Lieb Lemma \cite{Brezis} and Proposition \ref{ppro1}, the following holds:
\begin{equation}\label{l1.19}
\lim\limits_{n\rightarrow\infty} \left(a+bD\right)\|\nabla v_{n}\|^{2}_{2}=\lim\limits_{n\rightarrow\infty} \int_{\R^{N}}(I_{\mu}\ast|v_{n}|^{2_{\mu}^{*}})|v_{n}|^{2_{\mu}^{*}}.
\end{equation}

Let $\lim\limits_{n\rightarrow\infty} \|\nabla v_{n}\|^{2}_{2}=t$ and $\lim\limits_{n\rightarrow\infty}\int_{\R^{N}}(I_{\mu}\ast|v_{n}|^{2_{\mu}^{*}})|v_{n}|^{2_{\mu}^{*}}=m,$ thus from \eqref{l1.19}, we have
\begin{equation*}
at+bt^{\theta}\leq (at+bD)t= m.
\end{equation*}
This subsequently gives 
$2(abt^{\theta+1})^{1/2}\leq m$, i.e.,
$\lim\limits_{n\rightarrow\infty}\|\nabla v_{n}\|^{2{(\theta+1)}}_{2}\leq \frac{m^2}{4ab}$. Using the definition of $S_{HL}$, we have
$S_{HL}^{(\theta+1)}\left( \int_{\R^{N}}(I_{\mu}\ast|v_{n}|^{2_{\mu}^{*}})|v_{n}|^{2_{\mu}^{*}}\right)^\frac{(\theta+1)}{2_{\mu}^{*}}\leq \frac{m^2}{4ab}.$ Hence
\begin{equation}\label{l1.11}
m\geq\left( S_{HL}^{\theta+1}4ab\right) ^\frac{2_{\mu}^{*}}{22_{\mu}^{*}-(\theta+1)} \ \ \ \text{   or   } \ \ \  m=0.
\end{equation}
\textbf{Case 1:} If $m\geq\left( S_{HL}^{\theta+1}4ab\right) ^\frac{2_{\mu}^{*}}{22_{\mu}^{*}-(\theta+1)}$.
Resuming the fact that $\{u_n\}$ is a Palais Smale sequence together with \eqref{l1.101}, \eqref{l1.19} and \eqref{l1.11}, we deduce
\begin{align*}
l=\lim\limits_{n\rightarrow\infty}J_{\alpha}(u_{n})&=J_{\alpha}(u)+\lim\limits_{n\rightarrow\infty}\left( \frac{a}{2}\|\nabla v_{n}\|^{2}_{2}+\frac{bD}{2\theta}\|\nabla v_{n}\|^{2}_{2} -\frac{1}{22_{\mu}^{*}}\int_{\R^{N}}(I_{\mu}\ast|v_{n}|^{2_{\mu}^{*}})|v_{n}|^{2_{\mu}^{*}}\right)\\
&= J_{\alpha}(u)+\lim\limits_{n\rightarrow\infty}\left( \frac{m}{22_{\mu}^{*}\theta}(2_{\mu}^{*}-\theta)+\frac{a}{2}\frac{\theta-1}{\theta} \|\nabla v_{n}\|^{2}_{2}\right)\\
&\geq J_{\alpha}(u)+\frac{m}{2}\left( \frac{2_{\mu}^{*}-\theta}{2_{\mu}^{*}\theta}\right)\\
&\geq J_{\alpha}(u)+\left( S_{HL}^{\theta+1}4ab\right) ^\frac{2_{\mu}^{*}}{22_{\mu}^{*}-(\theta+1)}\left( \frac{2_{\mu}^{*}-\theta}{22_{\mu}^{*}\theta}\right), 
\end{align*}
this implies
$$l-\left( S_{HL}^{\theta+1}4ab\right) ^\frac{2_{\mu}^{*}}{22_{\mu}^{*}-(\theta+1)}\left( \frac{2_{\mu}^{*}-\theta}{22_{\mu}^{*}\theta}\right)\geq J_{\alpha}(u),$$ hence, alternative \textit{(1)} holds.

\textbf{Case II:} If $\lim\limits_{n\rightarrow\infty}\int_{\R^{N}}(I_{\mu}\ast|v_{n}|^{2_{\mu}^{*}})|v_{n}|^{2_{\mu}^{*}}=m=0$. Then by \eqref{l1.19}, $\lim\limits_{n\rightarrow\infty} \|\nabla v_{n}\|_{2}=0$, i.e., $\nabla u_{n}\rightarrow \nabla u$ in $L^{2}(\R^{N})$ or $ u_{n}\rightarrow u$ in $D^{1,2}(\R^{N})$
which further gives $u_{n}\rightarrow  u$ in $L^{2^*}(\R^{N})$.

Testing \eqref{l2.3} and \eqref{l1.3} with $v=u_{n}-u$, we have
\begin{align*}\label{l1.4}
 \left(a+b\|\nabla u_{n}\|_{2}^{2(\theta-1)}\right)& \int_{\R^{N}} |\nabla (u_{n}-u)|^{2}   -\lambda \int_{\R^{N}} | u_{n}-u|^{2}\\
&-\alpha \int_{\R^{N}}[(I_{\mu}\ast|u_{n}|^{q})|u_{n}|^{q-2}u_{n}-(I_{\mu}\ast|u|^{q})|u|^{q-2}u] (u_{n}-u)\\
&-\int_{\R^{N}} [(I_{\mu}\ast|u_{n}|^{2_{\mu}^{*}})|u_{n}|^{2_{\mu}^{*}-2}u_{n}-(I_{\mu}\ast|u|^{2_{\mu}^{*}})|u|^{2_{\mu}^{*}-2}u] (u_{n}-u)=o_{n}(1).
\end{align*}
In above, passing limit  and using the facts that  $u_{n}\rightarrow u$ in $D^{1,2}(\R^{N})$,  
 $u_{n}\rightarrow u$ in $L^{r}(\R^{N})$ for each $2<r\leq 2^{*}$ gives us
$$\lim_{n\rightarrow\infty}\int_{\R^{N}} | u_{n}-u|^{2} =0,$$
 
 $u_{n}\rightarrow u$ in $H^{1}(\R^{N}).$ Thus, the assertion \textit{(2)} holds.
\QED\end{proof} 
\section{Mixed critical case}
This section is devoted to the proof of Theorems \ref{t1} and \ref{t2}. First, ${S_c}$ .  Then, we deploy the min-max principle \cite[Theorem 5.2]{ghoussoub1993duality} to obtain the existence and multiplicity of solutions. 


\begin{lemma}\label{l4.1}
Let $2_{\mu,*}<q<A^{*},  B^{*}<p\leq2_{\mu}^{*}$ and $0<\alpha<\alpha_{1}$, where $\alpha_{1}$ is defined in \eqref{mu1}. Then $\mathfrak{P}^{0}_{\alpha}=\phi$ and $\mathfrak{P_{\alpha}}$ is a smooth manifold of codimension $2$ in $H^{1}(\R^{N})$. Moreover, all the critical points  of $J_{\alpha}|_{\mathfrak{P_{\alpha}}}$ are also the critical point of $J_{\alpha}|_{S_{c}}$.
\end{lemma}

\begin{proof}
Let on the contrary that $\mathfrak{P}^{0}_{\alpha}\neq\phi$ then there exists $u\in\mathfrak{P}^{0}_{\alpha}$, i.e., $P_{\alpha}(u)=0$ and $(E_{u})''(0)=0$. Precisely, we have
\begin{equation}\label{l1.13}
a\|\nabla u\|^{2}_{2}+b\|\nabla u\|^{2\theta}_{2}-\delta_{q}\alpha\int_{\R^{N}}(I_{\mu}\ast|u|^{q})|u|^{q}-\delta_{p}\int_{\R^{N}}(I_{\mu}\ast|u|^{p})|u|^{p}=0
\end{equation} 
and 
\begin{equation}\label{l1.14}
2a\|\nabla u\|^{2}_{2}+2\theta b\|\nabla u\|^{2\theta}_{2}-\alpha2q\delta_{q}^2\int_{\R^{N}}(I_{\mu}\ast|u|^{q})|u|^{q}-2p\delta_{p}^2\int_{\R^{N}}(I_{\mu}\ast|u|^{p})|u|^{p}=0.
\end{equation} 
By \eqref{l1.13}, \eqref{l1.14} and Gagliardo-Nirenberg inequality \eqref{GN1} (using the definition of $S_{HL}$, if $p=2_{\mu}^{*}$), we obtain
\begin{align}\label{l1.15}
\notag a(p\delta_{p}-1)\|\nabla u\|^{2}_{2}+b(p\delta_{p}-\theta)\|\nabla u\|^{2\theta}_{2}&=\alpha \delta_{q}(p\delta_{p}-q\delta_{q})\int_{\R^{N}}(I_{\mu}\ast|u|^{q})|u|^{q}\\
&\leq \alpha \delta_{q}(p\delta_{p}-q\delta_{q}) C_{q}\|\nabla u\|_{2}^{2q\delta_{q}}c^{2q(1-\delta_{q})},
\end{align}
and
\begin{align}\label{l1.161}
\notag a(1-q\delta_{q})\|\nabla u\|^{2}_{2}+b(\theta-q\delta_{q})\|\nabla u\|^{2\theta}_{2}&= \delta_{p}(p\delta_{p}-q\delta_{q})\int_{\R^{N}}(I_{\mu}\ast|u|^{p})|u|^{p}\\
&\leq  \delta_{p}(p\delta_{p}-q\delta_{q}) C_{p}\|\nabla u\|_{2}^{2p\delta_{p}}c^{2p(1-\delta_{p})}.
\end{align} 
By \eqref{l1.15} and \eqref{l1.161}, the following estimate holds
$$\left( \dfrac{b(\theta-q\delta_{q})}{\delta_{p}(p\delta_{p}-q\delta_{q}) C_{p}c^{2p(1-\delta_{p})}}\right) ^{\frac{1}{2p\delta_{p}-2\theta}}\leq\|\nabla u\|_{2}\leq \left( \dfrac{\alpha \delta_{q}(p\delta_{p}-q\delta_{q}) C_{q}c^{2q(1-\delta_{q})}}{b(p\delta_{p}-\theta)}\right) ^{\frac{1}{2\theta-2q\delta_{q}}}.$$
It implies
$$\left( \dfrac{b(\theta-q\delta_{q})}{\delta_{p}(p\delta_{p}-q\delta_{q}) C_{p}c^{2p(1-\delta_{p})}}\right) ^{\frac{\theta-q\delta_{q}}{p\delta_{p}-\theta}}\dfrac{b(p\delta_{p}-\theta) }{\delta_{q}(p\delta_{p}-q\delta_{q}) C_{q}c^{2q(1-\delta_{q})}}\leq \alpha,$$
which is a contradiction to the fact that $\alpha<\alpha_{1}$. Hence, our assumption is wrong, consequently $\mathfrak{P}^{0}_{\alpha}\neq\phi$. Using the similar arguments as in \cite[Lemma 5.2]{soave2020normalized1}, $\mathfrak{P_{\alpha}}$ is a smooth manifold of codimension $2$ in $H^{1}(\R^{N})$.

Let $u\in \mathfrak{P_{\alpha}}$ be a critical point for $J_{\alpha}|_{\mathfrak{P_{\alpha}}}$  then by the Lagrange multiplier rule, there exist $\lambda,\chi\in \R$ such that 
$$J'_{\alpha}(u)(v)-\lambda\int_{\R^{N}}uv-\chi P'_{\alpha} (u)(v)=0, \ v\in H^{1}(\R^{N}),$$ 
this implies that $u$ solves the following equation:
\begin{align*}
-\left(a(1-2\chi)+b(1-2\chi \theta)\|\nabla u\|_{2}^{2(\theta-1)}\right) &\Delta u  -\lambda u+\alpha (2\chi q\delta_{q}-1) (I_{\mu}\ast|u|^{q})|u|^{q-2}u\\
&+(2\chi p\delta_{p}-1)(I_{\mu}\ast|u|^{p})|u|^{p-2}u=0
\end{align*}
then the corresponding Pohozaev identity is:
\begin{align*}
\notag a(1-2\chi)\|\nabla u\|_{2}^{2}+b(1-2\chi \theta)\|\nabla u\|_{2}^{2\theta}&+\alpha\delta_{q}(2\chi q\delta_{q}-1)\int_{\R^{N}}(I_{\mu}\ast|u|^{q})|u|^{q}\\
&+\delta_{p}(2\chi p\delta_{p}-1)\int_{\R^{N}}(I_{\mu}\ast|u|^{p})|u|^{p}=0.
\end{align*}
As, $u\in \mathfrak{P_{\alpha}}$ and $u$ does not belongs to $\mathfrak{P}^{0}_{\alpha}$, we have
\begin{align*}
 \chi(a\|\nabla u\|_{2}^{2}+b\theta\|\nabla u\|_{2}^{2\theta}-\alpha q\delta_{q}^{2}\int_{\R^{N}}(I_{\mu}\ast|u|^{q})|u|^{q}- p\delta_{p}^{2}\int_{\R^{N}}(I_{\mu}\ast|u|^{p})|u|^{p})=0,
\end{align*}
which implies,
$$J'_{\alpha}(u)(v)-\lambda\int_{\R^{N}}uv=0, \ v\in H^{1}(\R^{N}).$$
Hence, $u$ is the critical point of $J_{\alpha}|_{S_{c}}$. 
\QED\end{proof}

Consider, the functional 
\begin{align*} J_{\alpha}(u)&=\frac{a}{2}\|\nabla u\|^{2}_{2}+\frac{b}{2\theta}\|\nabla u\|^{2\theta}_{2}-\frac{\alpha}{2q}\int_{\R^{N}}(I_{\mu}\ast|u|^{q})|u|^{q}-\frac{1}{2p}\int_{\R^{N}}(I_{\mu}\ast|u|^{p})|u|^{p}\\
&\geq \frac{a}{2}\|\nabla u\|^{2}_{2}+\frac{b}{2\theta}\|\nabla u\|^{2\theta}_{2}-\frac{\alpha}{2q}C_{q}c^{2q(1-\delta_{q})}\|\nabla u\|_{2}^{2q\delta_{q}}-\frac{1}{2p}C_{p}c^{2p(1-\delta_{p})}\|\nabla u\|_{2}^{2p\delta_{p}}.
\end{align*}
We define, the function $g:\R^{+}\rightarrow \R$ such that
$$g(t)=\frac{a}{2}t^{2}+\frac{b}{2\theta}t^{2\theta}-\frac{\alpha}{2q}C_{q}c^{2q(1-\delta_{q})}t^{2q\delta_{q}}-\frac{1}{2p}C_{p}c^{2p(1-\delta_{p})}t^{2p\delta_{p}}.$$
We can see that, $J_{\alpha}(u)\geq g(\|\nabla u\|_{2})$.
\begin{lemma}\label{l4.2}
Let $a_{1},a_{2},a_{3},a_{4}\in(0,\infty)$, $p_{1}\in(2\theta,\infty)$ and $q_{1}\in(0,2)$. If 
\begin{equation}\label{l2.1}
\left(t_{1}^{\frac{2\theta-q_{1}}{p_{1}-2\theta}}-t_{1}^{\frac{p_1-q_{1}}{p_{1}-2\theta}}\right)\left[  \dfrac{a_{1}}{a_{4}}\left(\dfrac{a_{2}}{a_{3}} \right)^{\frac{2-q_{1}}{p_{1}-2\theta}}+\dfrac{1}{a_{4}}\dfrac{a_{2}^{\frac{p_{1}-q_{1}}{p_{1}-2\theta}}}{a_{3}^{\frac{2\theta-q_{1}}{p_{1}-2\theta}}}  \right]>1 
\end{equation}
 where $t_{1}=\frac{2\theta(2\theta-q_{1})(2\theta-2)}{p_{1}(p_{1}-q_{1})(p_{1}-2)}$ then $h(t)=a_{1}t^{2}+a_{2}t^{2\theta}-a_{3}t^{p_{1}}-a_{4}t^{q_{1}}$ has a global strict maxima of positive level and a local strict minima at negative level on $[0,\infty)$.
\end{lemma}
\begin{proof}
The proof is similar to that of the \cite[Lemma 4.3]{li2022normalized}. For the sake of brevity, we omit the details.
\QED\end{proof}

\begin{lemma}\label{l4.3}
Let $2_{\mu,*}<q<A^{*},  B^{*}<p\leq2_{\mu}^{*}$ and $0<\alpha<\alpha_{2}$, where $\alpha_{2}$ is defined in \eqref{mu2}. Then $g$ has a global strict maxima of positive level and a local strict minima at negative level. Moreover, there exist $0<t_{0}<t_{1}$ (dependent on $c$ and $\alpha$) such that $g(t_{0})=g(t_{1})=0$ and $g(t)>0$ if and only if $t\in(t_{0},t_{1})$. 
\end{lemma}

\begin{proof}
Let, $a_{1}=\frac{a}{2}$, $a_{2}=\frac{b}{2\theta}$, $a_{3}=\frac{1}{2p}C_{p}c^{2p(1-\delta_{p})}$, $a_{4}=\frac{\alpha}{2q}C_{q}c^{2q(1-\delta_{q})}$, $p_{1}=2p\delta_{p}$ and $q_{1}=2q\delta_{q}$ in the Lemma \ref{l4.2}. To employ the Lemma \ref{l4.2}, we need to verify \eqref{l2.1}.
Consider,

\begin{align*}
&\left(   t_{1}^{\frac{2\theta-q_{1}}{p_{1}-2\theta}}-t_{1}^{\frac{p_1-q_{1}}{p_{1}-2\theta}}\right)\left[  \dfrac{a_{1}}{a_{4}}\left(\dfrac{a_{2}}{a_{3}} \right)^{\frac{2-q_{1}}{p_{1}-2\theta}}+\dfrac{1}{a_{4}}\dfrac{a_{2}^{\frac{p_{1}-q_{1}}{p_{1}-2\theta}}}{a_{3}^{\frac{2\theta-q_{1}}{p_{1}-2\theta}}}  \right]\\
& =\kappa\left[\dfrac{aq}{\alpha C_{q}c^{2q(1-\delta_{q})}}\left(\dfrac{bp}{\theta C_{p}c^{2p(1-\delta_{p})}} \right)^{\frac{1-q\delta_{q}}{p\delta_{p}-\theta}}+\dfrac{2q}{\alpha C_{q}c^{2q(1-\delta_{q})}}\dfrac{\frac{b}{2\theta}^{\frac{p\delta_{p}-q\delta_{q}}{p\delta_{p}-\theta}}}{\left(\frac{1}{2p}C_{p}c^{2p(1-\delta_{p})} \right) ^{\frac{\theta-q\delta_{q}}{p\delta_{p}-\theta}}}  \right]\\
&  =\kappa\frac{q}{C_{q}\alpha}\left[\dfrac{a}{c^{2q(1-\delta_{q})+\frac{2p(1-\delta_{p})(1-q\delta_{q})}{(p\delta_{p}-\theta)}}}\left(\dfrac{bp}{\theta C_{p}} \right)^{\frac{1-q\delta_{q}}{p\delta_{p}-\theta}}+\dfrac{2}{c^{2q(1-\delta_{q})+\frac{2p(1-\delta_{p})(\theta-q\delta_{q})}{(p\delta_{p}-\theta)}}}\dfrac{\left( \frac{b}{2\theta}\right) ^{\frac{p\delta_{p}-q\delta_{q}}{p\delta_{p}-\theta}}}{\left(\frac{C_{p}}{2p} \right) ^{\frac{\theta-q\delta_{q}}{p\delta_{p}-\theta}}}  \right]\\
&>1,
\end{align*}
as $\alpha<\alpha_{2}$, where $\kappa$ is defined in $\eqref{kappa}$. Thus, the conclusion follows from the Lemma \ref{l4.2}, i.e., $g$ has a global strict maxima of positive level and a local strict minima at negative level. Moreover, $g(0)=0$ and $g(t)\rightarrow -\infty$ as $t\rightarrow\infty$ follow the conclusion.
\QED\end{proof}

\begin{lemma}\label{l4.4}
Let $2_{\mu,*}<q<A^{*},  B^{*}<p\leq2_{\mu}^{*}$ and $0<\alpha<\min\{\alpha_{1},\alpha_{2}\}$, where $\alpha_{1}$ and $\alpha_{2}$ are defined in \eqref{mu1} and \eqref{mu2}, respectively. For any $u\in S_{c}$, the function $E_{u}$ has exactly two critical points $t^{1}_{u}<t^{3}_{u}\in\R$ and two zeros $t^{2}_{u}<t^{4}_{u}\in\R$ such that $t^{1}_{u}<t^{2}_{u}<t^{3}_{u}<t^{4}_{u}$. Moreover, we have the following:
\begin{enumerate}
\item[(1)] $t^{1}_{u}\star u\in \mathfrak{P_{\alpha}^{+}}$, $t^{3}_{u}\star u\in \mathfrak{P_{\alpha}^{-}}$ and $\mathfrak{P_{\alpha}}=\{t^{1}_{u}\star u,t^{3}_{u}\star u\}$.
\item[(2)] $\|\nabla (s\star u)\|_{2}\leq t_{0}$ ($t_{0}$ is as obtained in Lemma \ref{l4.3}) for all $s\leq t^{2}_{u}$ and $$J_{\alpha}(t^{3}_{u}\star u)=\max\{J_{\alpha}(t\star u):t\in\R\}>0.$$ Also, $$J_{\alpha}(t^{1}_{u}\star u)=\min\{J_{\alpha}(t\star u):t\in\R \text{ and } \|\nabla (s\star u)\|_{2}\leq t_{0}\}<0$$ and $E_{u}$ is strictly decreasing on $(t^{3}_{u},\infty)$.
\item[(3)] $u\in S_{c}\mapsto t^{1}_{u}\in \R$ and $u\in S_{c}\mapsto t^{3}_{u}\in \R$ are $C^{1}$ functional.
\end{enumerate}
\end{lemma}
\begin{proof}
Let, $u\in S_{c}$ then $u_{t}=t^{\frac{N}{2}}u(tx)\in S_{c}$ for  $t>0$. Consider, 
$$f(t)=J_{\alpha}(u_t)=\frac{a}{2}t^{2}\|\nabla u\|^{2}_{2}+\frac{b}{2\theta}t^{2\theta }\|\nabla u\|^{2\theta}_{2}-\frac{\alpha}{2q}t^{2\delta_{q}q}\int_{\R^{N}}(I_{\mu}\ast|u|^{q})|u|^{q}-\frac{1}{2p}t^{2\delta_{p}p}\int_{\R^{N}}(I_{\mu}\ast|u|^{p})|u|^{p}.$$
Take, $a_{1}=\frac{a}{2}\|\nabla u\|^{2}_{2}$, $a_{2}=\frac{b}{2\theta}\|\nabla u\|^{2\theta}_{2}$, $a_{3}=\frac{1}{2p}\int_{\R^{N}}(I_{\mu}\ast|u|^{p})|u|^{p}$, $a_{4}=\frac{\alpha}{2q}\int_{\R^{N}}(I_{\mu}\ast|u|^{q})|u|^{q}$, $p_{1}=2p\delta_{p}$ and $q_{1}=2q\delta_{q}$ in the Lemma \ref{l4.2}. We wish to deploy Lemma \ref{l4.2}. So precisely we have to verify \eqref{l2.1}.

Using Gagliardo-Nirenberg inequality \eqref{GN1} (using the definition of $S_{HL}$, if $p=2_{\mu}^{*}$), we get
\begin{align*}
\dfrac{a_{1}}{a_{4}}\left(\dfrac{a_{2}}{a_{3}} \right)^{\frac{2-q_{1}}{p_{1}-2\theta}}&=\dfrac{\frac{a}{2}\|\nabla u\|^{2}_{2}}{\frac{\alpha}{2q}\int_{\R^{N}}(I_{\mu}\ast|u|^{q})|u|^{q}}\left(\dfrac{\frac{b}{2\theta}\|\nabla u\|^{2\theta}_{2}}{\frac{1}{2p}\int_{\R^{N}}(I_{\mu}\ast|u|^{p})|u|^{p}} \right)^{\frac{2-2q\delta_{q}}{2p\delta_{p}-2\theta}}\\
&\geq \dfrac{aq\|\nabla u\|^{2}_{2}}{\alpha C_{q}\|\nabla u\|_{2}^{2q\delta_{q}}c^{2q(1-\delta_{q})}}\left(\dfrac{bp\|\nabla u\|^{2\theta}_{2}}{\theta C_{p}\|\nabla u\|_{2}^{2p\delta_{p}}c^{2p(1-\delta_{p})}} \right)^{\frac{1-q\delta_{q}}{p\delta_{p}-\theta}}\\
&\geq \dfrac{aq}{\alpha C_{q}c^{2q(1-\delta_{q})}}\left(\dfrac{bp}{\theta C_{p}c^{2p(1-\delta_{p})}} \right)^{\frac{1-q\delta_{q}}{p\delta_{p}-\theta}}\\
&\geq \dfrac{aq}{\alpha C_{q}c^{2q(1-\delta_{q})+\frac{2p(1-\delta_{p})(1-q\delta_{q})}{(p\delta_{p}-\theta)}}}\left(\dfrac{bp}{\theta C_{p}} \right)^{\frac{1-q\delta_{q}}{p\delta_{p}-\theta}}.
\end{align*}

Again using  Gagliardo-Nirenberg inequality \eqref{GN1} (using the definition of $S_{HL}$, if $p=2_{\mu}^{*}$), we obtain
\begin{align*}
\dfrac{1}{a_{4}}\dfrac{a_{2}^{\frac{p_{1}-q_{1}}{p_{1}-2\theta}}}{a_{3}^{\frac{2\theta-q_{1}}{p_{1}-2\theta}}} &=\dfrac{1}{\frac{\alpha}{2q}\int_{\R^{N}}(I_{\mu}\ast|u|^{q})|u|^{q}}\dfrac{\left( \frac{b}{2\theta}\|\nabla u\|^{2\theta}_{2}\right) ^{\frac{2p\delta_{p}-2q\delta_{q}}{2p\delta_{p}-2\theta}}}{\left(\frac{1}{2p}\int_{\R^{N}}(I_{\mu}\ast|u|^{p})|u|^{p}\right) ^{\frac{2\theta-2q\delta_{q}}{2p\delta_{p}-2\theta}}}\\
&\geq \dfrac{2q}{\alpha C_{q}\|\nabla u\|_{2}^{2q\delta_{q}}c^{2q(1-\delta_{q})}}\dfrac{\left( \frac{b}{2\theta}\|\nabla u\|^{2\theta}_{2}\right) ^{\frac{p\delta_{p}-q\delta_{q}}{p\delta_{p}-\theta}}}{\left(\frac{1}{2p}C_{p}\|\nabla u\|_{2}^{2p\delta_{p}}c^{2p(1-\delta_{p})}\right) ^{\frac{\theta-q\delta_{q}}{p\delta_{p}-\theta}}}\\
&\geq \dfrac{2q}{\alpha C_{q}c^{2q(1-\delta_{q})}}\dfrac{\left( \frac{b}{2\theta}\right) ^{\frac{p\delta_{p}-q\delta_{q}}{p\delta_{p}-\theta}}}{\left(\frac{1}{2p}C_{p}c^{2p(1-\delta_{p})}\right) ^{\frac{\theta-q\delta_{q}}{p\delta_{p}-\theta}}}\\
&\geq \dfrac{2q}{\alpha C_{q}{c^{2q(1-\delta_{q})+\frac{2p(1-\delta_{p})(\theta-q\delta_{q})}{(p\delta_{p}-\theta)}}}}\dfrac{\left( \frac{b}{2\theta}\right) ^{\frac{p\delta_{p}-q\delta_{q}}{p\delta_{p}-\theta}}}{\left(\frac{C_{p}}{2p}\right) ^{\frac{\theta-q\delta_{q}}{p\delta_{p}-\theta}}}\cdot
\end{align*}

Resuming the above information, one gets
\begin{align*}
&\left(   t_{1}^{\frac{2\theta-q_{1}}{p_{1}-2\theta}}-t_{1}^{\frac{p_1-q_{1}}{p_{1}-2\theta}}\right)\left[  \dfrac{a_{1}}{a_{4}}\left(\dfrac{a_{2}}{a_{3}} \right)^{\frac{2-q_{1}}{p_{1}-2\theta}}+\dfrac{1}{a_{4}}\dfrac{a_{2}^{\frac{p_{1}-q_{1}}{p_{1}-2\theta}}}{a_{3}^{\frac{2\theta-q_{1}}{p_{1}-2\theta}}}  \right]\\
&  =\kappa\frac{q}{C_{q}\alpha}\left[\dfrac{a}{c^{2q(1-\delta_{q})+\frac{2p(1-\delta_{p})(1-q\delta_{q})}{(p\delta_{p}-\theta)}}}\left(\dfrac{bp}{\theta C_{p}} \right)^{\frac{1-q\delta_{q}}{p\delta_{p}-\theta}}+\dfrac{2}{c^{2q(1-\delta_{q})+\frac{2p(1-\delta_{p})(\theta-q\delta_{q})}{(p\delta_{p}-\theta)}}}\dfrac{\left( \frac{b}{2\theta}\right) ^{\frac{p\delta_{p}-q\delta_{q}}{p\delta_{p}-\theta}}}{\left(\frac{C_{p}}{2p} \right) ^{\frac{\theta-q\delta_{q}}{p\delta_{p}-\theta}}}  \right]\\
&>1,
\end{align*}
as $\alpha<\alpha_{2}$,  where $\kappa$ is defined in $\eqref{kappa}$. From the Lemma \ref{l4.2}, $f(t)$ has a global strict maxima of positive level and a local strict minima at negative level on $[0,\infty)$. 

As, $E_{u}(s)=J_{\alpha}(s\star u)=f(e^s)\geq g(\|\nabla (s\star u)\|_{2})=g(e^{s}\|\nabla u\|_{2})$, by monotonicity of composite functions and Lemma \ref{l4.3}, we have $E_{u}$ has global strict maxima of positive level and local strict minima at negative level on $(-\infty,\infty)$ and $E_{u}$ is positive if and only if  $s\in \left(\log\left( \frac{t_{0}}{\|\nabla u\|_{2}}\right), \log\left( \frac{t_{1}}{\|\nabla u\|_{2}}\right) \right).$ Also, $E_{u}$ has exactly two critical points $t^{1}_{u}<t^{3}_{u}$ such that $t^{1}_{u}\in\left(-\infty, \log\left( \frac{t_{0}}{\|\nabla u\|_{2}}\right) \right)$ is the local minima at negative level and $t^{3}_{u}\in\left(\log\left( \frac{t_{0}}{\|\nabla u\|_{2}}\right), \log\left( \frac{t_{1}}{\|\nabla u\|_{2}}\right)\right) $ is the global maxima at positive level. By Remark \ref{r1}, we have $\mathfrak{P_{\alpha}}=\{t^{1}_{u}\star u,t^{3}_{u}\star u\}$.
As $t^{1}_{u}$ is the minima of $E_{u}$, we have $$(E_{t^{1}_{u}\star u})''(0)=(E_{u})''(t^{1}_{u})\geq 0$$ but $(E_{t^{1}_{u}\star u})''(0)\neq 0$, since $\mathfrak{P}_{\alpha}^{0}=\phi$. Hence, $(E_{t^{1}_{u}\star u})''(0)> 0$ this implies $t^{1}_{u}\star u\in \mathfrak{P_{\alpha}^{+}}$. Similarly, we have $t^{3}_{u}\star u\in \mathfrak{P_{\alpha}^{-}}$. By the monotonicity behavior, $E_{u}$ has exactly  two zeros $t^{2}_{u}<t^{4}_{u}$ such that $t^{1}_{u}<t^{2}_{u}<t^{3}_{u}<t^{4}_{u}$. By using the implicit function theorem, it can be seen that $u\mapsto t^{1}_{u}$ and $u\mapsto t^{3}_{u}$ are $C^{1}$ functions.
\QED\end{proof}
For any $r>0$, define the set 
$$\Upsilon_{r}=\{u\in S_{c}:\|\nabla u\|_{2}<r\}$$ and
$\overline{\Upsilon_{r}}$ denotes the closure of $\Upsilon_{r}$.

Set $m(c,\alpha)=\inf\limits_{u\in \Upsilon_{t_{0}}}J_{\alpha}(u)$, where $t_{0}$ is as obtained in Lemma \ref{l4.3}.

\begin{lemma}\cite[Lemma 4.7]{li2022normalized}\label{l3.5}
Let $2_{\mu,*}<q<A^{*},  B^{*}<p\leq2_{\mu}^{*}$ and $0<\alpha<\min\{\alpha_{1},\alpha_{2}\}$. The set $\mathfrak{P_{\alpha}^{+}}\subseteq \Upsilon_{t_{0}}$ and $\sup\limits_{\mathfrak{P_{\alpha}^{+}}}J_{\alpha}\leq0\leq\inf\limits_{\mathfrak{P_{\alpha}^{-}}}J_{\alpha}.$ Moreover, $$-\infty<m(c,\alpha)=\inf\limits_{\mathfrak{P_{\alpha}}}J_{\alpha}=\inf\limits_{\mathfrak{P_{\alpha}^{+}}}J_{\alpha}<0$$ 
and 
there exists $k>0$ independent of $c$ and $\alpha$ such that $m(c,\alpha)<\inf\limits_{\overline{\Upsilon_{t_{0}}}\backslash \Upsilon_{t_{0}-k}}J_{\alpha}$.
\end{lemma}

\begin{lemma}\cite[Lemma 4.8 and 4.9]{li2022normalized}
Let $2_{\mu,*}<q<A^{*},  B^{*}<p\leq2_{\mu}^{*}$ and $0<\alpha<\min\{\alpha_{1},\alpha_{2}\}$. Let us assume that $J_{\alpha}(u)<m(c,\alpha)$ then $t^{3}_{u}$ obtained in Lemma \ref{l4.4} is negative. Also, we have 
$$\breve{m}(c,\alpha)=\inf\limits_{\mathfrak{P_{\alpha}^{-}}}J_{\alpha}>0.$$
\end{lemma}


\textbf{Proof of the Theorem \ref{t1} (1): } Consider the minimizing sequence $\{w_{n}\}$ for $m(c,\alpha)$, we can assume that $\{w_{n}\}\subset S_{c,r}$ is radially decreasing sequence, if not then we can replace $|w_{n}|$ by $|w_{n}|^{*}$ ($|w_{n}|^{*}$ is the symmetric rearrangement of $|w_{n}|$) such that $J_{\alpha}(|w_{n}|^{*})\leq J_{\alpha}(|w_{n}|)$. By Lemma \ref{l4.4}, we have $t^{1}_{w_{n}}\star w_{n}\in \mathfrak{P_{\alpha}^{+}}$, $\|\nabla (t^{1}_{w_{n}}\star w_{n})\|_{2}\leq t_{0}$ and
$$J_{\alpha}(t^{1}_{w_{n}}\star w_{n})=\min\{J_{\alpha}(t\star w_{n}):t\in\R \text{ and } \|\nabla (t\star w_{n})\|_{2}\leq t_{0}\}\leq J_{\alpha}(w_n).$$ Hence, corresponding to the radially decreasing sequence $\{w_n\}$ we get the new minimizing sequence $\{v_{n}=t^{1}_{w_{n}}\star w_{n}\}$ for $m(c,\alpha)$ such that $v_{n}\in S_{c,r}\cap \mathfrak{P_{\alpha}^{+}}$ and  $P_{\alpha}(v_{n})=0$ for all $n\in \N$. By Lemma \ref{l3.5}, we obtain $\|\nabla v_{n}\|_{2}<t_{0}-k$ for all $n\in \N$ and for some $k>0$. Hence, by Ekeland’s variational principle, there exists a minimizing sequence $\{u_n\}\subset \Upsilon_{t_{0}}$ for  $m(c,\alpha)$ such that $\lim\limits_{n\rightarrow \infty}\|u_{n}-v_{n}\|= 0$ and it is a Palais Smale for $J_{\alpha}|_{S_{c}}$. The condition $\lim\limits_{n\rightarrow \infty}\|u_{n}-v_{n}\|= 0$ together with the boundedness of $\{v_{n}\}$ implies that $\|\nabla u_{n}\|_{2}<t_{0}-k$  and $\lim\limits_{n\rightarrow \infty}P_{\alpha}(u_{n})=0$. Therefore, applying Lemma \ref{l1} to the sequence $\{u_{n}\}$, there exists $u_{c,\alpha,loc}\in S_{c}$ such that $u_{n}\rightarrow u_{c,\alpha,loc}$ strongly (upto a subsequence) in $H^{1}(\R^{N})$ and $u_{c,\alpha,loc}$ is the radial solution of \eqref{1.1} for some $\lambda<0.$ It can be easily verify that $u_{c,\alpha,loc}$ is non-negative; hence, by strong maximum principle, $u_{c,\alpha,loc}$ is positive.

\indent It is clear that $u_{c,\alpha,loc}$ is the minimizer for  $m(c,\alpha)$ on the set $\Upsilon_{t_{0}}$.
By Remark \ref{r1}, set of all critical points for the $J_{\alpha}|_{S_{c}}$ is subset of $\mathfrak{P_{\alpha}}$ and by Lemma \ref{l3.5}, we have $J_{\alpha}(u_{c,\alpha,loc})=m(c,\alpha)=\inf\limits_{\mathfrak{P_{\alpha}}}J_{\alpha}$, thus $u_{c,\alpha,loc}$ is the ground state solution.

\indent Next, we will prove that any ground state solution for $J_{\alpha}|_{S_{c}}$ is a local minimizer of $J_{\alpha}$  on the set $\Upsilon_{t_{0}}$. Let $u$ be any ground state solution for $J_{\alpha}|_{S_{c}}$ then by Lemma \ref{l3.5}, we have $J_{\alpha}(u)=m(c,\alpha)=\inf\limits_{\mathfrak{P_{\alpha}}}J_{\alpha}<0< \inf\limits_{\mathfrak{P_{\alpha}^{-}}}J_{\alpha}$ which implies $u\in \mathfrak{P_{\alpha}^{+}}$. Again, by Lemma \ref{l3.5}, we have  $\mathfrak{P_{\alpha}^{+}}\subseteq \Upsilon_{t_{0}}$, thus $u$ is a local  minimizer of $J_{\alpha}$  on the set $\Upsilon_{t_{0}}$.

\textbf{Proof of the Theorem \ref{t1} (2): } Denote, $$J_{\alpha}^{\rho}=\{u\in S_{c}:J_{\alpha}(u)\leq \rho\}\cdot$$ 

Define the auxiliary $C^{1}$ function $\widehat{J}_{\alpha}(s,u):\R\times H^{1}(\R^{N})\rightarrow\R$ such that $$\widehat{J}_{\alpha}(s,u)=J_{\alpha}(s\star u)=\frac{a}{2}e^{2s}\|\nabla u\|^{2}_{2}+\frac{b}{2\theta}e^{2\theta s}\|\nabla u\|^{2\theta}_{2}-\frac{\alpha}{2q}e^{2\delta_{q}qs}\int_{\R^{N}}(I_{\mu}\ast|u|^{q})|u|^{q}-\frac{1}{2p}e^{2\delta_{p}ps}\int_{\R^{N}}(I_{\mu}\ast|u|^{p})|u|^{p}.$$ By \cite[Theorem 1.28]{willem2012minimax}, all the critical points of $\widehat{J}_{\alpha}|_{S_{c,r}}$ are also critical pints of $\widehat{J}_{\alpha}|_{S_{c}}$. 

Set,
$$\Gamma=\left\lbrace \gamma(t)=\left(\kappa(t),\xi(t)\right)\in C([0,1],\R\times S_{c,r}) ;\gamma(0)\in(0,\mathfrak{P_{\alpha}^{+}}),\gamma(1)\in(0,J_{\alpha}^{2m(c,\alpha)})\right\rbrace\cdot$$ 

By Lemma \ref{l4.4}, the minimax value 
$$\varsigma(c,\alpha)=\inf_{\gamma \in\Gamma}\max_{(s,u)\in\Gamma([0,1])}\widehat{J}_{\alpha}(s,u)$$ is a real number. As proved in the \cite[Theorem 1.1 (ii)]{li2022normalized}, we have
$$\varsigma(c,\alpha)=\inf_{\mathfrak{P_{\alpha}^{-}}\cap S_{c,r}}J_{\alpha}>0\geq \sup_{(\mathfrak{P_{\alpha}^{+}}\cup J_{\alpha}^{2m(c,\alpha)})\cap S_{c,r}}J_{\alpha}=\sup_{\left( (0,\mathfrak{P_{\alpha}^{+}})\cup (0,J_{\alpha}^{2m(c,\alpha)})\right) \cap (\R\times S_{c,r})}\widehat{J}_{\alpha}\cdot$$

Let $\gamma_{n}(t)=\left(\kappa_{n}(t),\xi_{n}(t)\right)$ be any minimizing sequence for $\varsigma(c,\alpha)$ such that $\kappa_{n}(t)=0$ and $\xi_{n}(t)\geq 0$ a.e. in $\R^{N}$, for all $t\in[0,1]$. With the help of $\gamma_{n}$, by using the standard arguments and as proved in \cite[Theorem 1.1 (ii)]{li2022normalized}, there exists a Palais-Smale sequence $\{u_{n}=r_{n}\star v_{n}\}\subset S_{c,r}$ for $J_{\alpha}|_{S_{c,r}}$ at level $\varsigma(c,\alpha)>0$ such that $P_{\alpha}(r_{n}\star v_{n})\rightarrow 0$. Thus, from the Lemma \ref{l1}, there exists $u_{c,\alpha,m}\in S_{c,r}$ such that  $u_{n}\rightarrow u_{c,\alpha,m}$ strongly (upto a subsequence) in $H^{1}(\R^{N})$. Moreover, $u_{c,\alpha,m}$ is the non-negative radial solution of \eqref{1.1} for some $\lambda_{c,\alpha,m}<0.$ By strong maximum principle, $u_{c,\alpha,loc}$ is positive.
\QED

To check the asymptotic behaviour of $m(c,\alpha)$ and $\varsigma(c,\alpha)$, we first discuss Problem \eqref{1.1} with $\alpha=0$.

\begin{lemma}\label{mu_zero1}
Let  $B^{*}<p<2_{\mu}^{*}$ and $\alpha=0$. Then $\mathfrak{P}^{0}_{\alpha}=\phi$ and $\mathfrak{P_{\alpha}}$ is a smooth manifold of codimension $2$ in $H^{1}(\R^{N})$. 
\end{lemma}
\begin{proof}
The proof is similar to the proof of the Lemma \ref{l4.1}.
\QED\end{proof}
\begin{lemma}\label{mu_zero2}
Let $B^{*}<p<2_{\mu}^{*}$ and $\alpha=0$. For any $u\in S_{c}$, the function $E_{u}$ has a unique maxima, say $t^{*}_{u}$ of positive level. Moreover, 
\begin{enumerate}
\item $E_{u}$ is strictly decreasing and concave on $(t^{*}_{u},\infty)$.
	\item $\mathfrak{P_{\alpha}}=\mathfrak{P_{\alpha}^{-}}$. Also, if $P_{\mu}(u)<0$, then $t^{*}_{u}<0$.
	\item The map $u\in S_{c}\mapsto t^{*}_{u}\in \R$ is a $C^{1}$ function.
\end{enumerate}
\end{lemma}
\begin{proof}
 Proceeding as the proof of \cite[Lemma 6.1]{soave2020normalized2}, we can easily obtain (1)-(3).
\QED\end{proof}

\begin{lemma}\label{mu_zero3}
Let  $B^{*}<p< 2_{\mu}^{*}$ and $\alpha=0$. Then $m(c,0)=\inf\limits_{u\in\mathfrak{P_{0}}}J_{0}(u)>0$. Moreover,
	$$0<\sup\limits_{u\in\overline{\Upsilon_{r}}}J_{0}(u)<m(c,0),$$
where $\Upsilon_{r}=\{u\in S_{c}:\|\nabla u\|_{2}^{2}<r\}$ for some sufficiently small $r>0$. Also, if $u\in\overline{\Upsilon_{r}}$ then $J_{0}(u)>0, \ P_{0}(u)>0.$
\end{lemma}
\begin{proof}
Let $u\in \mathfrak{P_{0}}$ be any arbitrary element. Then 
$$P_{0}(u)=a\|\nabla u\|^{2}_{2}+b\|\nabla u\|^{2\theta}_{2}=\delta_{p}\int_{\R^{N}}(I_{\mu}\ast|u|^{p})|u|^{p}\leq \delta_{p} C_{p}\|\nabla u\|_{2}^{2p\delta_{p}}c^{2p(1-\delta_{p})},$$
which implies $\|\nabla u\|_{2}\geq c_4>0$, for some $c_4>0$, since $p\delta_{p}>\theta>1$. Now, consider
\begin{align*}
m(c,0)=\inf\limits_{u\in\mathfrak{P_{0}}}J_{0}(u)&=\inf\limits_{u\in\mathfrak{P_{0}}}\left\lbrace  \frac{a}{2}\|\nabla u\|^{2}_{2}+\frac{b}{2\theta}\|\nabla u\|^{2\theta}_{2}-\frac{1}{2p}\int_{\R^{N}}(I_{\mu}\ast|u|^{p})|u|^{p}\right\rbrace\\
&=\inf\limits_{u\in\mathfrak{P_{0}}}\left\lbrace  \left( \frac{a}{2}-\frac{a}{2p\delta_{p}}\right) \|\nabla u\|^{2}_{2}+\left( \frac{b}{2\theta}-\frac{b}{2p\delta_{p}}\right)\|\nabla u\|^{2\theta}_{2}\right\rbrace\geq c_{5},
\end{align*}
for some $c_{5}>0$. Hence, $m(c,0)=\inf\limits_{u\in\mathfrak{P_{0}}}J_{0}(u)>0$.

Consider,
$$J_{0}(u)\geq \frac{b}{2\theta}\|\nabla u\|^{2\theta}_{2}-\frac{1}{2p}\int_{\R^{N}}(I_{\mu}\ast|u|^{p})|u|^{p}\geq \frac{b}{2\theta}\|\nabla u\|^{2\theta}_{2}-\frac{1}{2p}C_{p}\|\nabla u\|_{2}^{2p\delta_{p}}c^{2p(1-\delta_{p})}$$
and 
$$P_{0}(u)\geq b\|\nabla u\|^{2\theta}_{2}-\delta_{p}\int_{\R^{N}}(I_{\mu}\ast|u|^{p})|u|^{p}\geq b\|\nabla u\|^{2\theta}_{2} -\delta_{p} C_{p}\|\nabla u\|_{2}^{2p\delta_{p}}c^{2p(1-\delta_{p})}.$$
Note that, for any $u\in\overline{\Upsilon_{r}}$ for some sufficiently small $r>0$, we have 
	$\sup\limits_{u\in\overline{\Upsilon_{r}}}J_{0}(u)>0$ and
 $J_{0}(u)>0, \ P_{0}(u)>0.$ 
$$J_{0}(u)\leq\frac{a}{2}\|\nabla u\|^{2}_{2}+\frac{b}{2\theta}\|\nabla u\|^{2\theta}_{2}<m(c,0), \ \forall u\in \overline{\Upsilon_{r}},$$
 this completes the proof.
\QED\end{proof}

\begin{lemma}
Let  $B^{*}<p< 2_{\mu}^{*}$ and $\alpha=0$. Then there exists a positive radial critical point for $J_{\alpha}|_{S_{c}}$, say, $u_0$ such that
$$0<m_r(c,0)=m(c,0):=\inf_{\mathfrak{P}^{0}_{\alpha}}J_0(u)=J_{0}(u_0).$$
\end{lemma}
\begin{proof}
By using Lemmas \ref{mu_zero1}, \ref{mu_zero2},  \ref{mu_zero3} and following the similar technique presented in the Section $7$ of the paper \cite{soave2020normalized1}, we get the existence of positive radial critical point $u_0$ for $J_{\alpha}|_{S_{c}}$ at a mountain pass level $\varsigma(c,\alpha)>0$ such that $\inf\limits_{\mathfrak{P}^{0}_{\alpha}\cap S_{c,r}}J_0(u)=\varsigma(c,\alpha)$. Further, by standard rearrangement techniques and Lemma \ref{mu_zero2}, we have $m_r(c,0):=\inf\limits_{\mathfrak{P}^{0}_{\alpha}\cap S_{c,r}}J_0(u)=\inf_{\mathfrak{P}^{0}_{\alpha}}J_0(u).$
\QED\end{proof}

\begin{lemma}\label{mu_zero4}
Let $2_{\mu,*}<q<A^{*},  B^{*}<p< 2_{\mu}^{*}$ and $0<\alpha<\min\{\alpha_{1},\alpha_{2}\}$. Then, we have
$$\inf_{u\in\mathfrak{P_{\alpha}^{-}}\cap S_{c,r}}J_{\alpha}(u)=\inf_{u\in S_{c,r}}\max_{s\in \R}J_{\alpha}(s\star u) \text{ and } \inf_{u\in\mathfrak{P_{0}^{-}}\cap S_{c,r}}J_{\alpha}(u)=\inf_{u\in S_{c,r}}\max_{s\in \R}J_{0}(s\star u).$$
 
Moreover, for any $0\leq\alpha_{3}<\alpha_{4}<\min\{\alpha_{1},\alpha_{2}\}$, we have $$\varsigma(c,\alpha_{4})\leq \varsigma(c,\alpha_{3})\leq m(c,0).$$
\end{lemma}
\begin{proof}
The proof is similar to that of  \cite[Lemma 4.15 and 4.16]{li2022normalized}.
\QED\end{proof}

\textbf{Proof of Theorem \ref{t1} (3): }
 By Lemma \ref{l4.3}, we get $t_{0}\rightarrow 0$ as $\alpha\rightarrow 0^{+}$, therefore, we have $\|\nabla u_{c,\alpha,loc}\|_{2}<t_{0}\rightarrow 0.$ Now, consider
\begin{align*}
0>m(c,\alpha)\geq &\frac{a}{2}\|\nabla u_{c,\alpha,loc}\|^{2}_{2}+\frac{b}{2\theta}\|\nabla u_{c,\alpha,loc}\|^{2\theta}_{2}-\frac{\alpha}{2q}C_{q}c^{2q(1-\delta_{q})}\|\nabla u_{c,\alpha,loc}\|_{2}^{2q\delta_{q}}\\
&-\frac{1}{2p}C_{p}c^{2p(1-\delta_{p})}\|\nabla u_{c,\alpha,loc}\|_{2}^{2p\delta_{p}} \rightarrow 0,
\end{align*}
consequently $m(c,\alpha)\rightarrow 0.$

$(4)$: Define the set $\{u_{c,\alpha,m} \ : \ 0<\alpha<\mathring{\alpha}\}$, where $\mathring{\alpha}$ is sufficiently small. Since $P_{\alpha}(u_{c,\alpha,m})= 0$ and using the Lemma \ref{mu_zero4}, we have
\begin{align*}
m(c,0)\geq \varsigma(c,\alpha)&=J_{\alpha}(u_{c,\alpha,m})=\frac{a}{2}\|\nabla u_{c,\alpha,m}\|^{2}_{2}+\frac{b}{2\theta}\|\nabla u_{c,\alpha,m}\|^{2\theta}_{2}-\frac{\alpha}{2q}\int_{\R^{N}}(I_{\mu}\ast|u_{c,\alpha,m}|^{q})|u_{c,\alpha,m}|^{q}\\
-\frac{1}{2p}&\left(\frac{a}{\delta_{p}}\|\nabla u_{c,\alpha,m}\|^{2}_{2}+\frac{b}{\delta_{p}}\|\nabla u_{c,\alpha,m}\|^{2\theta}_{2}-\frac{\alpha\delta_{q}}{\delta_{p}}\int_{\R^{N}}(I_{\mu}\ast|u_{c,\alpha,m}|^{q})|u_{c,\alpha,m}|^{q}\right)\\
& \geq \left(\frac{a}{2}- \frac{a}{2p\delta_{p}}\right)\|\nabla u_{c,\alpha,m}\|^{2}+\left(\frac{b}{2\theta}- \frac{b}{2p\delta_{p}}\right)\|\nabla u_{c,\alpha,m}\|^{2\theta}_{2}\\
&-\frac{\alpha}{2q}\left(1-\frac{q\delta_{q}}{p\delta_{p}}\right)C_{q}\|\nabla u_{c,\alpha,m}\|_{2}^{2q\delta_{q}}c^{2q(1-\delta_{q})},
\end{align*}
which implies $\{u_{c,\alpha,m}\}$ is bounded in $H^{1}(\R^{N}).$ Since $H^{1}_{rad}(\R^{N})\hookrightarrow L^{r}(\R^{N})$ compactly for all $r\in\left( 2,2^{*}\right)$, there exists $\{u_{0}\}\in H^{1}_{rad}(\R^{N})$ such that 
$u_{c,\alpha,m}\rightharpoonup u_{0}$ in $H^{1}(\R^{N})$ and $u_{c,\alpha,m}\rightarrow u_{0}$ in $L^{r}(\R^{N})$ as $\alpha\rightarrow 0$. As $u_{c,\alpha,m}$ is positive for all $\alpha$, we have  $u_{c,\alpha,m}\rightarrow u_{0}\geq 0$ a.e. in $\R^{N}$. 

Also, we know, $u_{c,\alpha,m}\in S_{c,r}$ satisfies the following equation: 
\begin{align}\label{test_1}
\notag \left(a+b\|\nabla u_{c,\alpha,m}\|_{2}^{2(\theta-1)}\right) &\int_{\R^{N}} \nabla u_{c,\alpha,m} \ \nabla v  -\lambda_{c,\alpha,m}\int_{\R^{N}} u_{c,\alpha,m} \ v \\
&-\alpha \int_{\R^{N}}(I_{\mu}\ast|u_{c,\alpha,m}|^{q})|u_{c,\alpha,m}|^{q-2}u_{c,\alpha,m} \ v \\
\notag &-\int_{\R^{N}} (I_{\mu}\ast|u_{c,\alpha,m}|^{p})|u_{c,\alpha,m}|^{p-2}u_{c,\alpha,m} \ v=0, \ \forall \ v\in H^{1}_{rad}(\R^{N}),
\end{align}
for some $\lambda_{c,\alpha,m}<0$.
Further, taking into account, $P_{\alpha}(u_{c,\alpha,m})= 0$, boundedness of $\{u_{c,\alpha,m}\}$ and  \eqref{GN1}, one gets
\begin{align*}
\notag \lambda_{c,\alpha,m}c^{2}&=\alpha(\delta_{q}-1)\int_{\R^{N}}(I_{\mu}\ast|u_{c,\alpha,m}|^{q})|u_{c,\alpha,m}|^{q}
+(\delta_{p}-1)\int_{\R^{N}}(I_{\mu}\ast|u_{c,\alpha,m}|^{p})|u_{c,\alpha,m}|^{p}\\
\notag &\leq \alpha(\delta_{q}-1)c_{6}c^{2q(1-\delta_{q})}+(\delta_{p}-1)c_{7}c^{2p(1-\delta_{p})},
\end{align*}
for some $c_{6},c_{7}>0.$
By above and using the fact $\delta_{p},\delta_{q}<1$, we obtain $\lambda_{c,\alpha,m} \rightarrow  \lambda_{0}\leq 0$ when $\alpha\rightarrow 0$  with $ \lambda_{0}c^{2}=(\delta_{p}-1)\int_{\R^{N}}(I_{\mu}\ast|u_{0}|^{p})|u_{0}|^{p}$, which implies $\lambda_{0}=0$ if and only if $u_{0}\equiv 0$.

Claim: $\lambda_{0}<0$. Since, $u_{c,\alpha,m}\rightharpoonup u_{0}$ in $H^{1}(\R^{N})$ when $\alpha\rightarrow 0$ implies that $u_{0}$ is the solution for the equation
\begin{align}\label{test_2}
-\left(a+bD\right) \Delta u_{0} =\lambda u_{0}+(I_{\mu}\ast|u_{0}|^{2_{\mu}^{*}})|u_{0}|^{2_{\mu}^{*}-2}u_{0} \ \hbox{in} \  \R^{N}, 
\end{align}

where $D=\lim\limits_{\alpha\rightarrow 0} \|\nabla u_{c,\alpha,m}\|^{2(\theta-1)}_{2}\geq \|\nabla u_{0}\|^{2(\theta-1)}_{2}$. 
By Lemma \ref{mu_zero4}, we have
\begin{align*}
0<\varsigma(c,\mathring{\alpha})\leq\lim\limits_{\alpha\rightarrow 0}&\varsigma(c,\alpha)=\lim\limits_{\alpha\rightarrow 0} J_{\alpha}(u_{c,\alpha,m})\\
&=\lim\limits_{\alpha\rightarrow 0}\left[ -\frac{b(\theta-1)}{2\theta}\|\nabla u_{c,\alpha,m}\|^{2\theta}_{2}-\left( \frac{\alpha }{2q}-\frac{\alpha \delta_{q}}{2}\right) \int_{\R^{N}}(I_{\mu}\ast|u_{c,\alpha,m}|^{q})|u_{c,\alpha,m}|^{q}\right.\\
&\left.-\left( \frac{1}{2p}-\frac{ \delta_{p}}{2}\right) \int_{\R^{N}}(I_{\mu}\ast|u_{c,\alpha,m}|^{p})|u_{c,\alpha,m}|^{p}\right]\\
&\leq -\frac{b(\theta-1)}{2\theta}\|\nabla u_{0}\|^{2\theta}_{2}+\left( \frac{ \delta_{p}}{2}-\frac{1}{2p}\right) \int_{\R^{N}}(I_{\mu}\ast|u_{0}|^{p})|u_{0}|^{p},
\end{align*}
which implies $\frac{b(\theta-1)}{2\theta}\|\nabla u_{0}\|^{2\theta}_{2}\leq \left( \frac{ \delta_{p}}{2}-\frac{1}{2p}\right) \int_{\R^{N}}(I_{\mu}\ast|u_{0}|^{p})|u_{0}|^{p}\leq \left( \frac{ \delta_{p}}{2}-\frac{1}{2p}\right) C_{p}\|\nabla u_{0}\|_{2}^{2p\delta_{p}}c^{2p(1-\delta_{p})}$, hence $u_{0}\neq 0$. Therefore, $\lambda_{0}<0$ and $D>0$. Hence by strong maximum principle $u_{0}$ is positive. Subtracting \eqref{test_2} from \eqref{test_1} with the test function $u_{c,\alpha,m}-u_{0}$, we obtain
$$\left(a+bD\right)\int_{\R^{N}} |\nabla (u_{c,\alpha,m}-u)|^{2}  -\lambda \int_{\R^{N}} | u_{c,\alpha,m}-u|^{2} \rightarrow 0$$ 
which yields $u_{c,\alpha,m}\rightarrow u$ in $H^{1}(\R^{N})$ as $\alpha\rightarrow 0$. Also, we have $m(c,0)\leq J_{0}(u_{0})$. Using the fact that  $u_{c,\alpha,m}\rightarrow u$ in $D^{1,2}(\R^{N})$, we obtain $J_{0}(u_{0})=\lim\limits_{\alpha\rightarrow 0} J_{\alpha}(u_{c,\alpha,m})=\lim\limits_{\alpha\rightarrow 0}\varsigma(c,\alpha)\leq m(c,0).$ Therefore, $u_{0}$ is the ground state solution of $J_{0}|_{S_{c}}$.
\QED

\textbf{Proof of the Theorem \ref{t2}(1): } Proceeding as the proof of Theorem \ref{t1} \textit{$(1)$}, we have a radially decreasing Palais Smale sequence $\{u_{n}\}$ such that $\|\nabla u_{n}\|_{2}<t_{0}-k$  and $\lim\limits_{n\rightarrow \infty}P_{\alpha}(u_{n})=0$.
Hence, $\{u_{n}\}$ hold all the assumptions of Lemma \ref{l2}.\\
Claim:  $u_{n}\rightarrow u_{c,\alpha,loc}$ strongly in $H^{1}(\R^{N})$ for some $u_{c,\alpha,loc}\in H^{1}(\R^{N})$.
Let on the contrary that the alternative (1) of the Lemma \ref{l2} is true, i.e., 
$u_{n}\rightharpoonup u_{c,\alpha,loc}\neq 0$ weakly in $H^{1}(\R^{N})$ but not strongly then for some $\lambda<0$, $u_{c,\alpha,loc}$ is the solution for the equation:
\begin{equation}\label{t4.9}
-\left(a+bD\right) \Delta u =\lambda u+\alpha (I_{\mu}\ast|u|^{q})|u|^{q-2}u+(I_{\mu}\ast|u|^{2_{\mu}^{*}})|u|^{2_{\mu}^{*}-2}u \ \hbox{in} \  \R^{N} \end{equation}
and
\begin{equation}\label{t2.1}
l\geq \tilde{J}_{\alpha}(u_{c,\alpha,loc})+\left( S_{HL}^{\theta+1}4ab\right) ^\frac{2_{\mu}^{*}}{22_{\mu}^{*}-(\theta+1)}\left( \frac{2_{\mu}^{*}-\theta}{22_{\mu}^{*}\theta}\right),
\end{equation}
where  $D=\lim\limits_{n\rightarrow\infty} \|\nabla u_{n}\|^{2(\theta-1)}_{2}$.
The corresponding Pohozaev identity is
\begin{equation}\label{t4.10}
Q_{\alpha}(u_{c,\alpha,loc})=\left(a+bD\right)\|\nabla u_{c,\alpha,loc}\|^{2}_{2}-\delta_{q}\alpha\int_{\R^{N}}(I_{\mu}\ast|u_{c,\alpha,loc}|^{q})|u_{c,\alpha,loc}|^{q}-\int_{\R^{N}}(I_{\mu}\ast|u_{c,\alpha,loc}|^{2_{\mu}^{*}})|u|^{2_{\mu}^{*}}=0.
\end{equation}

By using \ref{t4.10}, $\|u_{c,\alpha,loc}\|_{2}\leq c$ and the fact that $u_{c,\alpha,loc}$ is the solution of \eqref{t4.9}, we have
\begin{align*} \tilde{J}_{\alpha}(u_{c,\alpha,loc})&=\left( \frac{a}{2}+\frac{Db}{2\theta}\right) \|\nabla u_{c,\alpha,loc}\|^{2}_{2}-\frac{\alpha}{2q}\int_{\R^{N}}(I_{\mu}\ast|u_{c,\alpha,loc}|^{q})|u_{c,\alpha,loc}|^{q}\\
&-\frac{1}{22_{\mu}^{*}}\int_{\R^{N}}(I_{\mu}\ast|u_{c,\alpha,loc}|^{2_{\mu}^{*}})|u_{c,\alpha,loc}|^{2_{\mu}^{*}}\\
&=\left( \frac{a}{2}+\frac{Db}{2\theta}\right) \|\nabla u_{c,\alpha,loc}\|^{2}_{2}-\frac{\alpha}{2q}\int_{\R^{N}}(I_{\mu}\ast|u_{c,\alpha,loc}|^{q})|u_{c,\alpha,loc}|^{q}\\
&-\frac{1}{22_{\mu}^{*}}\left[ \left(a+bD\right)\|\nabla u_{c,\alpha,loc}\|^{2}_{2}-\delta_{q}\alpha\int_{\R^{N}}(I_{\mu}\ast|u_{c,\alpha,loc}|^{q})|u_{c,\alpha,loc}|^{q}\right] \\
&=a\left( \frac{2_{\mu}^{*}-1}{22_{\mu}^{*}}\right)\|\nabla u_{c,\alpha,loc}\|^{2}_{2}+ b\left( \frac{2_{\mu}^{*}-\theta}{22_{\mu}^{*}\theta}\right)\|\nabla u_{c,\alpha,loc}\|^{2\theta}_{2}\\
&- \alpha\left( \frac{2_{\mu}^{*}-q\delta_{q}}{2q2_{\mu}^{*}}\right)\int_{\R^{N}}(I_{\mu}\ast|u_{c,\alpha,loc}|^{q})|u_{c,\alpha,loc}|^{q}\\
&\geq b\left( \frac{2_{\mu}^{*}-\theta}{22_{\mu}^{*}\theta}\right)\|\nabla u_{c,\alpha,loc}\|^{2\theta}_{2}- \alpha\left( \frac{2_{\mu}^{*}-q\delta_{q}}{2q2_{\mu}^{*}}\right)C_{q}\|\nabla u_{c,\alpha,loc}\|_{2}^{2q\delta_{q}}c^{2q(1-\delta_{q})}=f(\|\nabla u_{c,\alpha,loc}\|_{2}),
\end{align*}
where
$f(t)=b\left( \frac{2_{\mu}^{*}-\theta}{22_{\mu}^{*}\theta}\right)t^{2\theta}- \alpha\left( \frac{2_{\mu}^{*}-q\delta_{q}}{2q2_{\mu}^{*}}\right)C_{q}c^{2q(1-\delta_{q})}t^{2q\delta_{q}}$. Therefore, $f'(t^{*})=0$, where $$t^{*}=\left[ \dfrac{\alpha\delta_{q}(2_{\mu}^{*}-q\delta_{q})C_{q}c^{2q(1-\delta_{q})}}{b(2_{\mu}^{*}-\theta)}\right]^\frac{1}{2\theta-2q\delta_{q}}.$$ 

We can see that $f(t)\longrightarrow\infty$ as $t\longrightarrow\infty$, thus we have \begin{align}\label{t2.2}
 \notag \min_{t\geq 0}f(t)=f(t^{*})&=-b(t^{*})^{2\theta}\left(\frac{\theta-q\delta_{q}}{\theta q\delta_{q}} \right)\left(\frac{2_{\mu}^{*}-\theta}{2 2_{\mu}^{*}} \right)\\
&> -\left( S_{HL}^{\theta+1}4ab\right) ^\frac{2_{\mu}^{*}}{22_{\mu}^{*}-(\theta+1)}\left( \frac{2_{\mu}^{*}-\theta}{22_{\mu}^{*}\theta}\right),
\end{align}
since, $\alpha<\alpha_{3},$ where $\alpha_{3}$ is defined in \eqref{mu3}. 

As $\tilde{J}_{\alpha}(u_{c,\alpha,loc})\geq f(\|\nabla u_{c,\alpha,loc}\|_{2})$, using \eqref{t2.1} and \eqref{t2.2}, we get
$$0>l\geq\tilde{J}_{\alpha}(u_{c,\alpha,loc})+\left( S_{HL}^{\theta+1}4ab\right) ^\frac{2_{\mu}^{*}}{22_{\mu}^{*}-(\theta+1)}\left( \frac{2_{\mu}^{*}-\theta}{22_{\mu}^{*}\theta}\right)>0,$$
which is a contradiction. Hence, our supposition is wrong; thus  $u_{n}\rightarrow u_{c,\alpha,loc}$ strongly in $H^{1}(\R^{N})$, $u_{c,\alpha,loc}\in S_{c}$ and $u_{c,\alpha,loc}$ is a minimizer for  $m(c,\alpha)$ on the set $\Upsilon_{t_{0}}$ and it is a solution of \eqref{1.1} for some $\lambda<0.$ It can be easily verify that $u_{c,\alpha,loc}$ is non-negative and radially decreasing; hence, by the strong maximum principle $u_{c,\alpha,loc}$ is positive. Proceeding as the proof of Theorem \ref{t1} (1), we get $u_{c,\alpha,loc}$ is a ground state solution of $J_{\alpha}|_{S_{c}}$ and any ground state solution of $J_{\alpha}|_{S_{c}}$ is a local minimizer of $J_{\alpha}$  on the set $\Upsilon_{t_{0}}$.

\textbf{Proof of the Theorem \ref{t2}(2): }  By Lemma \ref{l4.3}, we get $t_{0}\rightarrow 0$ as $\alpha\rightarrow 0^{+}$, therefore, we have $\|\nabla u_{c,\alpha,loc}\|_{2}<t_{0}\rightarrow 0.$ Now, consider
\begin{align*}
0>m(c,\alpha)\geq &\frac{a}{2}\|\nabla u_{c,\alpha,loc}\|^{2}_{2}+\frac{b}{2\theta}\|\nabla u_{c,\alpha,loc}\|^{2\theta}_{2}-\frac{\alpha}{2q}C_{q}c^{2q(1-\delta_{q})}\|\nabla u_{c,\alpha,loc}\|_{2}^{2q\delta_{q}}\\
&-\frac{1}{2p}C_{p}c^{2p(1-\delta_{p})}\|\nabla u_{c,\alpha,loc}\|_{2}^{2p\delta_{p}} \rightarrow 0,
\end{align*}
consequently $m(c,\alpha)\rightarrow 0.$ 
\QED

\section{$L^{2}-$ supercritical case}
In this section, first, we discuss the existence of a solution for the Problem \eqref{1.1} when $B^{*}<q<p<2_{\mu}^{*}$, i.e., $L^{2}-$ supercritical and Sobolev subcritical case. Later, we will discuss the particular case of the Problem \ref{1.1} when $N=3$, $\theta=2$, $2_{\mu}^{*}=4$ and $\mu=2$ with $L^{2}-$ supercritical and Sobolev critical exponent, i.e., $B^{*}<q<p=2_{\mu}^{*}$.


\begin{lemma}\label{sl1}
	Let $a_{1},a_{2},a_{3},a_{4}\in(0,\infty)$ and $p_{1},q_{1}\in(2\theta,\infty)$. Then $h(t)=a_{1}t^{2}+a_{2}t^{2\theta}-a_{3}t^{p_{1}}-a_{4}t^{q_{1}}$ has a unique maxima of positive level on $[0,\infty)$.
\end{lemma}
\begin{proof}
	 By using the same assertions as in  \cite[Lemma 5.1]{li2022normalized}, we get the desired result. \QED
\end{proof}
\begin{lemma}\label{sl2}
Let $B^{*}<q<p\leq 2_{\mu}^{*}$. For any $u\in S_{c}$, the function $E_{u}$ has a unique maxima, say $t^{*}_{u}$ of positive level. Moreover, the following holds: 
\begin{enumerate}
	\item $E_{u}$ is strictly decreasing on $(t^{*}_{u},\infty)$. Also, if  $t^{*}_{u}<0$ then $P_{\alpha}(u)<0$.
	\item $\mathfrak{P_{\alpha}}=\mathfrak{P_{\alpha}^{-}}$. Also, if $P_{\alpha}(u)<0$, then $t^{*}_{u}<0$.
	\item The map $u\in S_{c}\mapsto t^{*}_{u}\in \R$ is a  $C^{1}$ function.
\end{enumerate}
\end{lemma}
\begin{proof}
By Lemma \ref{sl1}, $E_{u}$ has unique maxima, say $t^{*}_{u}$ of positive level. Proceeding as in  proof of \cite[Lemma 6.1]{soave2020normalized2}, we can easily obtain (1)-(3).
\QED\end{proof}

\begin{lemma}\label{sl3}
Let  $B^{*}<q<p\leq 2_{\mu}^{*}$. Then $m(c,\alpha)=\inf\limits_{u\in\mathfrak{P_{\alpha}}}J_{\alpha}(u)>0$. Also, there exists $r>0$ (sufficiently small) such that $$0<\sup\limits_{u\in\overline{\Upsilon_{r}}}J_{\alpha}(u)<m(c,\alpha),$$
where $\Upsilon_{r}=\{u\in S_{c}:\|\nabla u\|_{2}^{2}<r\}$. Moreover, if $u\in\overline{\Upsilon_{r}}$ then $J_{\alpha}(u)>0, \ P_{\alpha}(u)>0.$
\end{lemma}
\begin{proof}
The proof is similar to that of Lemma \ref{mu_zero3}.
\QED\end{proof}


\subsection{$L^{2}-$ supercritical and Sobolev critical: Particular case}

In this subsection, we study the Problem \eqref{1.1} when $N=3$, $\theta=2$, $2_{\mu}^{*}=4$ and $\mu=2$ with $L^{2}-$ supercritical and Sobolev critical exponent, i.e., $\frac{8}{3}<q<p=4$.

\begin{lemma}\label{l5.4}
Let $\frac{8}{3}<q<p=4$ and $\frac{a^{2}}{4}-\frac{b^{3}S_{HL}^{4}}{27}>0$. Assume $\{u_n\}\subseteq S_{c,r}$ be a Palais-Smale sequence for the functional $J_{\alpha}|_{S_{c}}$ of level $l\neq 0$ with 
 \begin{equation}\label{l1.12}
 l<\left( \frac{b \Lambda^{2} S_{HL}^{2} }{8}+\frac{3a\Lambda S_{HL}}{8}\right) 
 \end{equation}
and $P_{\alpha}(u_n)\rightarrow 0$ as $n\rightarrow\infty$, where
 \begin{equation}\label{Lambda}
 	\Lambda=\left(\frac{aS_{HL}}{2}+\sqrt{\frac{a^{2}S_{HL}^{2}}{4}-\frac{b^{3}S_{HL}^{6}}{27}}\right)^{\frac{1}{3}}+\left(\frac{aS_{HL}}{2}-\sqrt{\frac{a^{2}S_{HL}^{2}}{4}-\frac{b^{3}S_{HL}^{6}}{27}}\right)^{\frac{1}{3}}. 
 	\end{equation}

 Then, one of the following holds:
\begin{enumerate}
\item[(i)] Either $u_{n}\rightharpoonup u\neq 0$ weakly in $H^{1}(\R^{N})$ but not strongly then $u$ is the solution of the following equation for some $\lambda<0$:
$$
-\left(a+bD\right) \Delta u =\lambda u+\alpha (I_{\mu}\ast|u|^{q})|u|^{q-2}u+(I_{\mu}\ast|u|^{2_{\mu}^{*}})|u|^{2_{\mu}^{*}-2}u \ \hbox{in} \  \R^{N} $$
and
$$l-\left( \frac{b \Lambda^{2} S_{HL}^{2} }{8}+\frac{3a\Lambda S_{HL}}{8}\right)\geq \tilde{J}_{\alpha}(u),$$ where  $D=\lim\limits_{n\rightarrow\infty} \|\nabla u_{n}\|^{2(\theta-1)}_{2}$.
\item[(ii)] Or $u_{n}\rightarrow u$ strongly in $H^{1}(\R^{N}).$ Also, $u\in S_{c}, J_{\alpha}(u)=l$ and $u$ solves \eqref{1.1} for some $\lambda<0$.
\end{enumerate} 
 Moreover, $u$ is the radial solution of \eqref{1.1} for some $\lambda<0.$
\end{lemma}
\begin{proof}
Boundedness of $\{u_n\}$ and existence of Lagrange multipliers $\lambda_{n}\rightarrow \lambda\in\R$  follows from Lemma \ref{l1}. Next, we will prove that $\lambda\leq0$ and $u\neq 0$.
Taking into account  $P_{\alpha}(u_n)\rightarrow 0$, boundedness of $\{u_n\}$, \eqref{l1.4} and  \eqref{GN1}, one gets
\begin{align}\label{l1.6}
 \lambda_{n}c^{2}&=\alpha(\delta_{q}-1)\int_{\R^{N}}(I_{\mu}\ast|u_{n}|^{q})|u_{n}|^{q}+o_{n}(1)\\
\notag	&\leq \alpha(\delta_{q}-1)C_{q}\|\nabla u_{n}\|_{2}^{2q\delta_{q}}c^{2q(1-\delta_{q})}+o_{n}(1)\\
	\notag &\leq \alpha(\delta_{q}-1)c_{8}c^{2q(1-\delta_{q})}+o_{n}(1),
\end{align}
for some $c_{8}>0.$
Taking $n\rightarrow\infty$ and using the fact $\delta_{q}<1$, we obtain $\lambda c^{2}\leq 0$ which implies $\lambda\leq 0.$

\textbf{Claim:} $\boldsymbol{\lambda<0.}$

 Let, if possible on the contrary, $\lambda=0$ then by \eqref{l1.6}, we have 
\begin{equation}\label{l1.10}
\lim\limits_{n\rightarrow\infty}\int_{\R^{N}}(I_{\mu}\ast|u_{n}|^{q})|u_{n}|^{q}=0.
\end{equation}
By above and using the fact that $P_{\alpha}(u_n)\rightarrow 0$, we have
\begin{equation}\label{p.71}
\lim\limits_{n\rightarrow\infty}\left( a\|\nabla u_{n}\|^{2}_{2}+b\|\nabla u_{n}\|^{4}_{2}\right) =m,
\end{equation}
where $\lim\limits_{n\rightarrow\infty}\int_{\R^{N}}(I_{\mu}\ast|u_{n}|^{4})|u_{n}|^{4}=m.$
Therefore, by \eqref{a2} and \eqref{p.71}, we have
\begin{equation*}
 aS_{HL} m^{1/4}+bS_{HL}^{2}m^{1/2}\leq m.
\end{equation*}
Put $t=m^{1/4}$, then $aS_{HL} t+bS_{HL}^{2}t^{2}\leq t^{4}$ which implies
 either $t=0$ or  $t^{3}-bS_{HL}^{2}t-aS_{HL}\geq 0.$
 
 The equation $t^{3}-bS_{HL}^{2}t-aS_{HL}= 0$ is the depressed cubic equation, which has only one real root if the discriminant 
 $\frac{a^{2}S_{HL}^{2}}{4}-\frac{b^{3}S_{HL}^{6}}{27}>0$
 and it is equal to 
 $$t=m^{1/4}=\left(\frac{aS_{HL}}{2}+\sqrt{\frac{a^{2}S_{HL}^{2}}{4}-\frac{b^{3}S_{HL}^{6}}{27}}\right)^{\frac{1}{3}}+\left(\frac{aS_{HL}}{2}-\sqrt{\frac{a^{2}S_{HL}^{2}}{4}-\frac{b^{3}S_{HL}^{6}}{27}}\right)^{\frac{1}{3}}.$$
 The above equation implies that either $m=0$ or $m\geq \Lambda^{4}$, where
 $$\Lambda=\left(\frac{aS_{HL}}{2}+\sqrt{\frac{a^{2}S_{HL}^{2}}{4}-\frac{b^{3}S_{HL}^{6}}{27}}\right)^{\frac{1}{3}}+\left(\frac{aS_{HL}}{2}-\sqrt{\frac{a^{2}S_{HL}^{2}}{4}-\frac{b^{3}S_{HL}^{6}}{27}}\right)^{\frac{1}{3}}.$$
%
%
As $\{u_n\}$ is a Palais Smale sequence for the functional $J_{\alpha}$ of level $l$ and using \eqref{l1.10}, \eqref{p.71} and \eqref{a2}, we get
\begin{align*}
l=\lim\limits_{n\rightarrow\infty}J_{\alpha}(u_{n})&=\lim\limits_{n\rightarrow\infty}\left( \frac{a}{2}\|\nabla u_{n}\|^{2}_{2}+\frac{b}{4}\|\nabla u_{n}\|^{4}_{2} -\frac{1}{8}\int_{\R^{N}}(I_{\mu}\ast|u_{n}|^{4})|u_{n}|^{4}\right)\\
&\geq\lim\limits_{n\rightarrow\infty}\left( \frac{m}{8}+\frac{a}{4}\|\nabla u_{n}\|^{2}\right)\\
&\geq\left( \frac{m}{8}+\frac{a}{4}S_{HL}m^{1/4}\right)\\
&\geq\left( \frac{\Lambda^{4}}{8}+\frac{a\Lambda}{4}S_{HL}\right)\\
&\geq\left( \frac{b \Lambda^{2} S_{HL}^{2} }{8}+\frac{3a\Lambda S_{HL}}{8}\right)
\end{align*}
which is a contradiction to \eqref{l1.12}. It implies that our assumption is wrong.  Hence, $\lambda<0$ and $u\neq 0$. Further following the arguments as presented in the proof of Lemma \ref{l2} and above, we get the conclusion of the result.
\QED\end{proof} 


\begin{lemma}\label{l5.5}
Let $\frac{10}{3}<q<p=4$ and $\frac{a^{2}}{4}-\frac{b^{3}S_{HL}^{4}}{27}>0$. Then, we have 
$$m_{r}(c,\alpha)=\inf\limits_{u\in \mathfrak{P}_{c}\cap S_{c,r}}J_{\alpha}(u)<\left( \frac{b \Lambda^{2} S_{HL}^{2} }{8}+\frac{3a\Lambda S_{HL}}{8}\right),$$ where
$\Lambda$ is defined in \eqref{Lambda}.
\end{lemma}

\begin{proof}
	From \cite[Lemma 1.2]{gao2018brezis} we know that
	\begin{align*}
		U_\varepsilon(x)= S^{\frac{(N-\mu)(2-N)}{4(N-\mu+2)}}(C(N,\mu))^{\frac{2-N}{2(N-\mu+2)}}\left(\frac{\varepsilon}{\varepsilon^2+|x|^2}\right)^{\frac{N-2}{2}}, \; 0<\varepsilon<1
	\end{align*}
	are the minimizers of $S_{HL}$ in \eqref{a2}. 
	
	Now  define $\eta\in C_c^{\infty}(\mathbb{R}^N)$ such that $0\leq \eta\leq 1$ in $\mathbb{R}^N$, $\eta\equiv1$ in $B_\de(0)$ and $\eta\equiv 0 $ in $\mathbb{R}^N \setminus B_{2\de}(0)$ and $|\na \eta |< C$. Let $u_\varepsilon\in  H^1(\Om)$ be defined as  $u_\varepsilon(x)= \eta(x) U_\varepsilon(x)$ and $v_\varepsilon(x)=c\frac{u_\varepsilon(x)}{\|u_\varepsilon(x)\|_{2}}$. 
By Lemma \ref{sl2}, there exists a unique $t^{*}_{\mu}\in \R$ such that 
$$m_{r}(c,\mu)\leq J_{\alpha}(t^{*}_{\mu}\star v_{\varepsilon})=\max_{t\in \R}J_{\alpha}(t v_{\varepsilon})=\max_{t\in \R}E_{v_{\varepsilon}}(t).$$ 
We have the following estimates (see \cite{goel2019kirchhoff}), 
$\|\nabla u_{\varepsilon}\|^{2}_{2}=S_{HL}^{4/3}+O(\varepsilon)$, $\int_{\R^{N}}(I_{\mu}\ast|u_{\varepsilon}|^{4})|u_{\varepsilon}|^{4}\leq S_{HL}^{4/3}+O(\varepsilon^{3})$,  $\int_{\R^{N}}(I_{\mu}\ast|u_{\varepsilon}|^{4})|u_{\varepsilon}|^{4}\geq S_{HL}^{4/3}-O(\varepsilon^{3})$, $\int_{\R^{N}}(I_{\mu}\ast|u_{\varepsilon}|^{q})|u_{\varepsilon}|^{q}=O(\varepsilon^{4-q})$ and $\| u_{\varepsilon}\|^{2}_{2}=O(\varepsilon)$.

Define,
$$E_{v_{\varepsilon}}^{0}(t)=J_{0}(t\star v_{\varepsilon})=\frac{a}{2}e^{2t}\|\nabla v_{\varepsilon}\|^{2}_{2}+\frac{b}{4}e^{4 t}\|\nabla v_{\varepsilon}\|^{4}_{2}-\frac{1}{8}e^{8t}\int_{\R^{N}}(I_{\mu}\ast|v_{\varepsilon}|^{4})|v_{\varepsilon}|^{4}.$$
Consider the function $f(s)=c_{1}t+\frac{c_{2}}{2}s^{2}-\frac{c_{3}}{4}s^{4}$, where $c_{1}=\frac{a}{2}\|\nabla v_{\varepsilon}\|^{2}_{2}$, $c_{2}=\frac{b}{2}\|\nabla v_{\varepsilon}\|^{4}_{2}$ and $c_{3}=\frac{1}{2}\int_{\R^{N}}(I_{\mu}\ast|v_{\varepsilon}|^{4})|v_{\varepsilon}|^{4}$. Then, $f'(s)=c_{1}+c_{2}s-c_{3}s^{3}$, maximum point of $f$ are the zeros of $f'$. The equation $f'$ is a depressed cubic equation that has only one real root if the discriminant 
$\frac{c_{1}^{2}}{4c_{3}^{2}}-\frac{c_{2}^{3}}{27c_{3}^{3}}>0$
and it is equal to 
 $$s_{0}^{*}=\left(\frac{c_{1}}{2c_{3}}+\sqrt{\frac{c_{1}^{2}}{4c_{3}^{2}}-\frac{c_{2}^{3}}{27c_{3}^{3}}}\right)^{\frac{1}{3}}+\left(\frac{c_{1}}{2c_{3}}-\sqrt{\frac{c_{1}^{2}}{4c_{3}^{2}}-\frac{c_{2}^{3}}{27c_{3}^{3}}}\right)^{\frac{1}{3}}.$$ 
Thus, the unique maximum point of $E_{v_{\varepsilon}}^{0}$ is
\begin{align*}
e^{2t_{0}^{*}}=&\left(\frac{a\|\nabla v_{\varepsilon}\|^{2}_{2}}{2\int_{\R^{N}}(I_{\mu}\ast|v_{\varepsilon}|^{4})|v_{\varepsilon}|^{4}}+\sqrt{\frac{a^{2}\|\nabla v_{\varepsilon}\|^{4}_{2}}{4(\int_{\R^{N}}(I_{\mu}\ast|v_{\varepsilon}|^{4})|v_{\varepsilon}|^{4})^{2}}-\frac{b^{3}\|\nabla v_{\varepsilon}\|_{2}^{12}}{27(\int_{\R^{N}}(I_{\mu}\ast|v_{\varepsilon}|^{4})|v_{\varepsilon}|^{4})^{3}}}\right)^{\frac{1}{3}}\\
&+\left(\frac{a\|\nabla v_{\varepsilon}\|^{2}_{2}}{2\int_{\R^{N}}(I_{\mu}\ast|v_{\varepsilon}|^{4})|v_{\varepsilon}|^{4}}-\sqrt{\frac{a^{2}\|\nabla v_{\varepsilon}\|^{4}_{2}}{4(\int_{\R^{N}}(I_{\mu}\ast|v_{\varepsilon}|^{4})|v_{\varepsilon}|^{4})^{2}}-\frac{b^{3}\|\nabla v_{\varepsilon}\|_{2}^{12}}{27(\int_{\R^{N}}(I_{\mu}\ast|v_{\varepsilon}|^{4})|v_{\varepsilon}|^{4})^{3}}}\right)^{\frac{1}{3}}.
\end{align*}
This subsequently gives us
\begin{align*}
&\frac{e^{2t_{0}^{*}}c^{2}}{\|u_{\varepsilon}\|^{2}_{2}}=\left(\frac{a\|\nabla u_{\varepsilon}\|^{2}_{2}}{2\int_{\R^{N}}(I_{\mu}\ast|u_{\varepsilon}|^{4})|u_{\varepsilon}|^{4}}+\sqrt{\frac{a^{2}\|\nabla u_{\varepsilon}\|^{4}_{2}}{4(\int_{\R^{N}}(I_{\mu}\ast|u_{\varepsilon}|^{4})|u_{\varepsilon}|^{4})^{2}}-\frac{b^{3}\|\nabla u_{\varepsilon}\|_{2}^{12}}{27(\int_{\R^{N}}(I_{\mu}\ast|u_{\varepsilon}|^{4})|u_{\varepsilon}|^{4})^{3}}}\right)^{\frac{1}{3}}\\
&+\left(\frac{a\|\nabla u_{\varepsilon}\|^{2}_{2}}{2\int_{\R^{N}}(I_{\mu}\ast|u_{\varepsilon}|^{4})|u_{\varepsilon}|^{4}}-\sqrt{\frac{a^{2}\|\nabla u_{\varepsilon}\|^{4}_{2}}{4(\int_{\R^{N}}(I_{\mu}\ast|u_{\varepsilon}|^{4})|u_{\varepsilon}|^{4})^{2}}-\frac{b^{3}\|\nabla u_{\varepsilon}\|_{2}^{12}}{27(\int_{\R^{N}}(I_{\mu}\ast|u_{\varepsilon}|^{4})|u_{\varepsilon}|^{4})^{3}}}\right)^{\frac{1}{3}}.
\end{align*}

By using the estimates, we have
\begin{align*}
&\left(\frac{a\|\nabla u_{\varepsilon}\|^{2}_{2}}{2\int_{\R^{N}}(I_{\mu}\ast|u_{\varepsilon}|^{4})|u_{\varepsilon}|^{4}}+\sqrt{\frac{a^{2}\|\nabla u_{\varepsilon}\|^{4}_{2}}{4(\int_{\R^{N}}(I_{\mu}\ast|u_{\varepsilon}|^{4})|u_{\varepsilon}|^{4})^{2}}-\frac{b^{3}\|\nabla u_{\varepsilon}\|_{2}^{12}}{27(\int_{\R^{N}}(I_{\mu}\ast|u_{\varepsilon}|^{4})|u_{\varepsilon}|^{4})^{3}}}\right)^{\frac{1}{3}}\\
&\leq \left(\frac{a(S_{HL}^{4/3}+O(\varepsilon))}{2(S_{HL}^{4/3}-O(\varepsilon^{3}))}+\sqrt{\frac{a^{2}(S_{HL}^{4/3}+O(\varepsilon))^{2}}{4(S_{HL}^{4/3}-O(\varepsilon^{3}))^{2}}-\frac{b^{3}(S_{HL}^{4/3}+O(\varepsilon))^{6}}{27(S_{HL}^{4/3}+O(\varepsilon^{3}))^{3}}}\right)^{\frac{1}{3}}\\
&=\left(\frac{a(1+O(\varepsilon))}{2(1-O(\varepsilon^{3}))}+\sqrt{\frac{a^{2}(1+O(\varepsilon))^{2}}{4(1-O(\varepsilon^{3}))^{2}}-S_{HL}^{4}\frac{b^{3}(1+O(\varepsilon))^{6}}{27(1+O(\varepsilon^{3}))^{3}}}\right)^{\frac{1}{3}}\\
&=\left(\frac{a}{2}+\sqrt{\frac{a^{2}}{4}-S_{HL}^{4}\frac{b^{3}}{27}}\right)^{\frac{1}{3}}+O(\varepsilon^\frac{1}{3}).
\end{align*}
Consequently, we have
\begin{equation}\label{e1}
\frac{e^{2t_{0}^{*}}c^{2}}{\|u_{\varepsilon}\|^{2}_{2}}\leq \frac{\Lambda}{S_{HL}^{1/3}}+O(\varepsilon^\frac{1}{3}),
\end{equation}
where
$$\frac{\Lambda}{S_{HL}^{1/3}}=\left(\frac{a}{2}+\sqrt{\frac{a^{2}}{4}-S_{HL}^{4}\frac{b^{3}}{27}}\right)^{\frac{1}{3}}+\left(\frac{a}{2}-\sqrt{\frac{a^{2}}{4}-S_{HL}^{4}\frac{b^{3}}{27}}\right)^{\frac{1}{3}}.$$
Using the similar assertions as above, we get
\begin{equation}\label{e2}
\frac{e^{2t_{0}^{*}}c^{2}}{\|u_{\varepsilon}\|^{2}_{2}}\geq \frac{\Lambda}{S_{HL}^{1/3}}.
\end{equation}
By \eqref{e1} and \eqref{e2}, we have
\begin{align*}\sup_{t\in \R} E_{v_{\varepsilon}}^{0}(t)&=E_{v_{\varepsilon}}^{0}(t_{0}^{*})=\frac{a}{2}e^{2t_{0}^{*}}\|\nabla v_{\varepsilon}\|^{2}_{2}+\frac{b}{4}e^{4 t_{0}^{*}}\|\nabla v_{\varepsilon}\|^{4}_{2}-\frac{1}{8}e^{8t_{0}^{*}}\int_{\R^{N}}(I_{\mu}\ast|v_{\varepsilon}|^{4})|v_{\varepsilon}|^{4}\\
& =\frac{a}{2}\frac{e^{2t_{0}^{*}}c^{2}\|\nabla u_{\varepsilon}\|^{2}_{2}}{\|u_{\varepsilon}\|^{2}_{2}}+\frac{b}{4}\frac{e^{4t_{0}^{*}}c^{4}\|\nabla u_{\varepsilon}\|^{4}_{2}}{\|u_{\varepsilon}\|^{4}_{2}}-\frac{1}{8}\frac{e^{8t_{0}^{*}}c^{8}}{\|u_{\varepsilon}\|^{8}_{2}}\int_{\R^{N}}(I_{\mu}\ast|u_{\varepsilon}|^{4})|u_{\varepsilon}|^{4} \\
& \leq \frac{a}{2}\left( \frac{\Lambda}{S_{HL}^{1/3}}+O(\varepsilon^\frac{1}{3})\right)(S_{HL}^{4/3}+O(\varepsilon))+\frac{b}{4}\left( \frac{\Lambda}{S_{HL}^{1/3}}+O(\varepsilon^\frac{1}{3})\right)^{2}(S_{HL}^{4/3}+O(\varepsilon))^{2}\\
&-\frac{1}{8}\left( \frac{\Lambda}{S_{HL}^{1/3}}\right)^{4}(S_{HL}^{4/3}-O(\varepsilon^{3})) \\
&=\frac{a\Lambda S_{HL}}{2}+\frac{b\Lambda^{2} S_{HL}^{2}}{4}-\frac{\Lambda^{4}}{8}+O(\varepsilon^\frac{1}{3})\\
&=\frac{a\Lambda S_{HL}}{2}+\frac{b\Lambda^{2} S_{HL}^{2}}{4}-\frac{a\Lambda S_{HL}}{8}-\frac{b\Lambda^{2} S_{HL}^{2}}{8}+O(\varepsilon^\frac{1}{3})\\
&=\frac{3a\Lambda S_{HL}}{8}+\frac{b\Lambda^{2} S_{HL}^{2}}{8}+O(\varepsilon^\frac{1}{3}).
\end{align*}
On the other hand, we have $(E_{v_{\varepsilon}})'(t^{*}_{\mu})=P_{\alpha}(t^{*}_{\mu}\star v_{\varepsilon})=0$, since $J_{\alpha}(t^{*}_{\mu}\star v_{\varepsilon})=\max\limits_{t\in \R}E_{v_{\varepsilon}}(t)$ implies that $e^{2t_{\mu}^{*}}\leq e^{2t_{0}^{*}}$, therefore, by \eqref{e1}, we get
\begin{equation*}
e^{2t_{\mu}^{*}}\leq \frac{\|u_{\varepsilon}\|^{2}_{2}}{c^{2}}\left( \frac{\Lambda}{S_{HL}^{1/3}}+O(\varepsilon^\frac{1}{3})\right) 
\end{equation*}
and 
$$ae^{2t_{\mu}^{*}}\|\nabla v_{\varepsilon}\|^{2}_{2}+be^{4t_{\mu}^{*}}\|\nabla v_{\varepsilon}\|^{4}_{2}-\mu\delta_{q}e^{2\delta_{q}qt_{\mu}^{*}}\int_{\R^{N}}(I_{\mu}\ast|v_{\varepsilon}|^{q})|v_{\varepsilon}|^{q}=e^{8t_{\mu}^{*}}\int_{\R^{N}}(I_{\mu}\ast|v_{\varepsilon}|^{4})|v_{\varepsilon}|^{4}.$$
Therefore, we have
\begin{align*}e^{4t_{\mu}^{*}}&\geq \dfrac{b\|\nabla v_{\varepsilon}\|^{4}_{2}}{\int_{\R^{N}}(I_{\mu}\ast|v_{\varepsilon}|^{4})|v_{\varepsilon}|^{4}}-\dfrac{\mu\delta_{q}e^{(2\delta_{q}q-4)t_{\mu}^{*}}\int_{\R^{N}}(I_{\mu}\ast|v_{\varepsilon}|^{q})|v_{\varepsilon}|^{q}}{\int_{\R^{N}}(I_{\mu}\ast|v_{\varepsilon}|^{4})|v_{\varepsilon}|^{4}}\\
& \geq \dfrac{b\|\nabla u_{\varepsilon}\|^{4}_{2}\|u_{\varepsilon}\|^{4}_{2}}{c^{4}\int_{\R^{N}}(I_{\mu}\ast|u_{\varepsilon}|^{4})|u_{\varepsilon}|^{4}}-\dfrac{\mu\delta_{q}e^{(2\delta_{q}q-4)t_{\mu}^{*}}\int_{\R^{N}}(I_{\mu}\ast|u_{\varepsilon}|^{q})|u_{\varepsilon}|^{q}\|u_{\varepsilon}\|^{8-2q}_{2}}{{c^{^{8-2q}}\int_{\R^{N}}(I_{\mu}\ast|u_{\varepsilon}|^{4})|u_{\varepsilon}|^{4}}}\\
&\geq \dfrac{b\|\nabla u_{\varepsilon}\|^{4}_{2}\|u_{\varepsilon}\|^{4}_{2}}{c^{4}\int_{\R^{N}}(I_{\mu}\ast|u_{\varepsilon}|^{4})|u_{\varepsilon}|^{4}}-\dfrac{\mu\delta_{q}\int_{\R^{N}}(I_{\mu}\ast|u_{\varepsilon}|^{q})|u_{\varepsilon}|^{q}\|u_{\varepsilon}\|^{4-2q(1-\delta_{q})}_{2}}{{c^{4-2q(1-\delta_{q})}\int_{\R^{N}}(I_{\mu}\ast|u_{\varepsilon}|^{4})|u_{\varepsilon}|^{4}}}\left( \frac{\Lambda}{S_{HL}^{1/3}}+O(\varepsilon^\frac{1}{3}) \right)^{\delta_{q}q-2}\\
&\geq \dfrac{\|u_{\varepsilon}\|^{4}_{2}}{c^{4}}\left[ \dfrac{b\|\nabla u_{\varepsilon}\|^{4}_{2}}{\int_{\R^{N}}(I_{\mu}\ast|u_{\varepsilon}|^{4})|u_{\varepsilon}|^{4}}-\dfrac{\mu\delta_{q}c^{2q(1-\delta_{q})}\int_{\R^{N}}(I_{\mu}\ast|u_{\varepsilon}|^{q})|u_{\varepsilon}|^{q}}{\|u_{\varepsilon}\|^{2q(1-\delta_{q})}_{2}\int_{\R^{N}}(I_{\mu}\ast|u_{\varepsilon}|^{4})|u_{\varepsilon}|^{4}}\left( \frac{\Lambda}{S_{HL}^{1/3}}+O(\varepsilon^\frac{1}{3}) \right)^{\delta_{q}q-2}\right]\\
&\geq \dfrac{\|u_{\varepsilon}\|^{4}_{2}}{c^{4}}\left[ \dfrac{b(S_{HL}^{4/3}+O(\varepsilon))^{2}}{(S_{HL}^{4/3}+O(\varepsilon^{3}))}-\dfrac{\mu\delta_{q}c^{2q(1-\delta_{q})}(O(\varepsilon))^{4-q}}{O(\varepsilon)^{q(1-\delta_{q})}(S_{HL}^{4/3}-O(\varepsilon^{3}))}\left( \frac{\Lambda}{S_{HL}^{1/3}}+O(\varepsilon^\frac{1}{3}) \right)^{\delta_{q}q-2}\right]\\ 
&\geq \dfrac{\|u_{\varepsilon}\|^{4}_{2}}{c^{4}}\left[ bS_{HL}^{4/3}-\dfrac{\mu\delta_{q}c^{2q(1-\delta_{q})}(O(\varepsilon))^{4-q}}{S_{HL}^{4/3}O(\varepsilon)^{\frac{4-q}{2}}}\left( \frac{\Lambda}{S_{HL}^{1/3}}+O(\varepsilon^\frac{1}{3}) \right)^{\delta_{q}q-2}\right].
\end{align*}

Thus, we have $e^{t_{\mu}^{*}}\geq c_8\dfrac{\|u_{\varepsilon}\|_{2}}{c}$, for some $c_8>0$ and for sufficiently small $\varepsilon>0$. 
Finally, we calculate
\begin{align*}
\max_{t\in \R}E_{v_{\varepsilon}}(t)=E_{v_{\varepsilon}}(t_{\mu}^{*})&=E_{v_{\varepsilon}}^{0}(t_{0}^{*})-\mu\delta_{q}e^{2\delta_{q}qt_{\mu}^{*}}\int_{\R^{N}}(I_{\mu}\ast|v_{\varepsilon}|^{q})|v_{\varepsilon}|^{q}\\
& \leq \frac{3a\Lambda S_{HL}}{8}+\frac{b\Lambda^{2} S_{HL}^{2}}{8}+O(\varepsilon^\frac{1}{3})-\mu\delta_{q}C^{2q\delta_{q}}\dfrac{\|u_{\varepsilon}\|_{2}^{2q\delta_{q}}}{c^{2q\delta_{q}}}\int_{\R^{N}}(I_{\mu}\ast|v_{\varepsilon}|^{q})|v_{\varepsilon}|^{q}\\
& \leq \frac{3a\Lambda S_{HL}}{8}+\frac{b\Lambda^{2} S_{HL}^{2}}{8}+O(\varepsilon^\frac{1}{3})-\dfrac{\mu\delta_{q}C^{2q\delta_{q}}c^{2q(1-\delta_{q})}}{\|u_{\varepsilon}\|_{2}^{2q(1-\delta_{q})}}\int_{\R^{N}}(I_{\mu}\ast|u_{\varepsilon}|^{q})|u_{\varepsilon}|^{q}\\
& \leq \frac{3a\Lambda S_{HL}}{8}+\frac{b\Lambda^{2} S_{HL}^{2}}{8}+O(\varepsilon^\frac{1}{3})-\dfrac{\mu\delta_{q}C^{2q\delta_{q}}c^{2q(1-\delta_{q})}(O(\varepsilon))^{4-q}}{(O(\varepsilon))^{\frac{4-q}{2}}}\\
& \leq \frac{3a\Lambda S_{HL}}{8}+\frac{b\Lambda^{2} S_{HL}^{2}}{8}+O(\varepsilon^\frac{1}{3})-\mu\delta_{q}C^{2q\delta_{q}}c^{2q(1-\delta_{q})}(O(\varepsilon))^{\frac{4-q}{2}}\\
& \leq \frac{3a\Lambda S_{HL}}{8}+\frac{b\Lambda^{2} S_{HL}^{2}}{8}, 
\end{align*}
which completes the  proof.
\QED\end{proof}


We are now ready to prove Theorem \ref{t3} with the assistance of Lemmas \ref{l1} \ref{sl1}, \ref{sl2} and \ref{sl3}.

 \textbf{Proof of the Theorem \ref{t3}
 $(1)$: } Set,
$$\Gamma=\left\lbrace \gamma(t)=\left(\kappa(t),\xi(t)\right)\in C([0,1],\R\times S_{c,r}) ;\gamma(0)\in(0,\overline{\Upsilon_{r}}),\gamma(1)\in (0,J_{\alpha}^{0})\right\rbrace\cdot$$
then using the standard arguments as used in \cite[Theorem 1.6]{soave2020normalized1}, $J_{\alpha}|_{S_{c}}$ has a critical point via mountain pass theorem, say, $u_{c,\alpha,m}$ such that $J_{\alpha}(u_{c,\alpha,m})=\sigma(c,\alpha)>0$. Also,  $u_{c,\alpha,m}$ is a positive radial solution to \eqref{1.1} for some $\lambda_{c,\alpha,m}<0$. Moreover, $u_{c,\alpha,m}\in S_{c}$ is a ground state solution for $J_{\alpha}|_{S_{c}}$. 

\textbf{Proof of the Theorem \ref{t3}
$(2)$: } The proof is similar to the proof of the Theorem \ref{t1} $(4)$.    
\QED

\textbf{Proof of the Theorem \ref{t4}: }
Proof of the Theorem \ref{t4}] By Lemma \ref{l5.5}, we have $m_{r}(c,\alpha)<\left( \frac{b \Lambda^{2} S_{HL}^{2} }{8}+\frac{3a\Lambda S_{HL}}{8}\right).$ The proof is similar to the proof of the theorem \ref{t3} but replace Lemma \ref{l1} by lemma \ref{l5.4}.  
\bibliographystyle{siam}

\bibliography{siam}
\end{document}